\documentclass{article}
\usepackage{graphicx} % Required for inserting images
\usepackage{color}
\usepackage{amsthm}
\usepackage{graphicx}
\usepackage{amsmath}
\usepackage{amsfonts}
\usepackage{hyperref}
\usepackage{todonotes}
\usepackage{pgf,tikz,calc}
\usepackage{rotating}
   \usepackage{tabu}
\usepackage[affil-it]{authblk} % italics for affiliations
\usetikzlibrary{math,angles,quotes,trees, positioning}
\usepackage{slashbox}
\usepackage{float}
\usepackage[margin=2cm]{geometry}
\usepackage{amssymb}
\newtheorem{definition}{Definition}
\newtheorem{question}{Question}
\newtheorem{conj}{Conjecture}
\newtheorem{theorem}{Theorem}
\newtheorem{prop}[theorem]{Proposition}
\newtheorem{corollary}{Corollary}
\newtheorem{observation}{Observation}

\usepackage{pgfplots}
\pgfplotsset{compat=1.15}
\usepackage{mathrsfs}
\usetikzlibrary{arrows}

\title{Packing coloring of graphs with long paths}

\author[1]{Hanna Furma\'nczyk \thanks{Corresponding author: hanna.furmanczyk@ug.edu.pl}}
\author[2]{Didem G\"{o}z\"{u}pek \thanks{didem.gozupek@gtu.edu.tr}}
\author[3]{Sibel \"{O}zkan \thanks{s.ozkan@gtu.edu.tr}}

\affil[1]{Institute of Informatics, Faculty of Mathematics, Physics and Informatics, University of Gda\'nsk, Gda\'nsk, Poland}
\affil[2]{Department of Computer Engineering, Gebze Technical University, Türkiye}
\affil[3]{Department of Mathematics, Gebze Technical University, Türkiye}

\date{}

\begin{document}

\maketitle

\begin{abstract}
%    The packing coloring of a graph is a vertex coloring in which any two vertices receiving color $i$ are at a distance of at least $i+1$, and the smallest number of colors admitting such a coloring is the packing chromatic number. Goddard et al.\cite{goddard2008broadcast} showed that the packing chromatic numbers of paths and cycles are at most 3 and 4, respectively. In this paper, we first introduce \emph{path-aligned graph products}, which is a natural extension of paths and has unbounded diameter. Accordingly, we extend the result in \cite{goddard2008broadcast} by showing that the packing chromatic number remains bounded by a constant in several families of path-aligned cycle products and path-aligned complete products. We then focus on the packing chromatic number of caterpillars, which is another class of graphs with long induced paths. Sloper et al. \cite{sloper} proved that the packing chromatic number of caterpillars is bounded by $7$. In this paper, we provide a complete structural characterization of caterpillars with packing chromatic number at most 3. Furthermore, We pose several open research questions.

The packing coloring problem has diverse applications, including frequency assignment in wireless networks, resource distribution and facility location in smart cities and post-disaster management, as well as in biological diversity. Formally, the packing coloring of a graph is a vertex coloring in which any two vertices assigned color $i$ are at a distance of at least $i+1$, and the smallest number of colors admitting such a coloring is called the packing chromatic number. Goddard et al.~\cite{goddard2008broadcast} showed that the packing chromatic numbers of paths and cycles are at most 3 and 4, respectively. In this paper, we introduce \emph{path-aligned graph products}, a natural extension of paths with unbounded diameter. We extend the result of~\cite{goddard2008broadcast} by proving that the packing chromatic number remains bounded by a constant for several families of path-aligned cycle and path-aligned complete products. We then investigate the packing chromatic number of caterpillars, another class of graphs characterized by long induced paths. Sloper~\cite{sloper} proved that the packing chromatic number of caterpillars is at most 7; here, we provide a complete structural characterization of caterpillars with packing chromatic number at most 3. Finally, several open research questions are posed.

\end{abstract}

%\textit{\textbf{Keywords:} Optimization}


\section{Introduction}

A proper coloring of the vertices of a graph $G=(V(G), E(G))$ refers to assigning colors (integers) to the vertices in a way that ensures adjacent vertices receive different colors. The \emph{distance} $d_{G}(u, v)$ between two vertices $u$ and $v$ of a graph $G$ is defined as the number of edges in the shortest path connecting them. The \emph{diameter} of a graph $G$ is the maximum distance between any two of its vertices and it is denoted by $diam(G)$. Given a positive integer $i$, an $i$-\emph{packing} in $G$ is a subset $X \subseteq V(G)$ such that the distance between any two distinct vertices is greater than $i$. The \emph{packing coloring} of a graph refers to a proper coloring of its vertices, where color class $V_i$ forms $i$-packing, for any $i\in[k]=\{1, \ldots, k\}$. In other words, 
vertices assigned color $i$ must have a distance of at least $i+1$ from each other. The smallest number of colors admitting a packing coloring of a graph $G$ is called the \emph{packing chromatic number} of $G$ and it is denoted by \emph{$\chi_p(G)$}. In other words, \emph{$\chi_p(G)$} is the smallest integer $k$ such that the vertex set of $G$ can be partitioned into sets $V_1, V_2, \ldots, V_k$, where $V_i$ is an $i$-packing for each $i \in [k]$. Such a partition corresponds to a mapping $c:V(G)\rightarrow [k]$ such that $V_i=\{u \in V(G): c(u)=i\}$. The concept of packing coloring was introduced by Goddard et al. under the name of broadcast coloring \cite{goddard2008broadcast}. The term of packing coloring was coined by Brešar et al. \cite{packing1}.

Packing coloring has widespread applications. For example, in wireless networks, two radio stations
that are assigned the same frequency have to be sufficiently apart from each other to avoid
interference. The minimum distance required is related to the power of their broadcast signals. In particular,
all stations located at vertices in the $i$-packing $V_i$ are allowed to broadcast at the same frequency with a
power that will not allow the signals to interfere at distance $i$. 

Beyond telecommunications, packing coloring has applications in several domains related to 
resource distribution and facility location. Examples include the optimal placement of 
electric vehicle charging stations and emergency services in smart cities, 
where facilities of higher priority or higher power demand must be spaced further apart 
to ensure efficiency and safety. In particular, electric vehicle charging stations can be of various types, each with different speeds of charging and infrastructure deployment costs \cite{leijon2025evcharging}, hence having different distancing requirements. Similar situation exists for emergency response units, which may have different capacities in relation to the risk level and population of the services area, thereby leading to different spacing requirements. Similar models appear in the deployment of air-quality sensors 
and emergency shelters, where the type or sensitivity of the facility determines its minimal separation 
from others of the same kind. In particular, high sensitivity air quality sensors need to be placed at proper distances to capture a comprehensive pollution map. In post-disaster management, different types of emergency shelters, such as medical, educational, and family shelters, have distinct distancing requirements. Another application of packing coloring is in biological diversity, where different
species in a certain area require different amounts of territory.

%\subsection{Motivation}

Graphs with long paths arise in numerous scenarios in urban areas. In transportation networks, long paths represent major thoroughfares, highways, or transit routes that span significant distances within a city. Long paths can also model the layout of water and gas pipelines, as well as the transmission lines that carry electricity from power plants to substations and distribution networks in electricity grids. For instance, caterpillar graphs model main roads and branching side streets.

Goddard et al.\cite{goddard2008broadcast} showed that the packing chromatic numbers of paths and cycles are at most 3 and 4, respectively. So, a natural research question is the following: Which natural extensions of paths and cycles preserve the property that the packing chromatic number remains bounded by a constant, independent of the graph diameter? In this paper, we address this research problem by introducing path-aligned graph products, a new class of graph products defined herein, and we specifically investigate path-aligned cycle products and path-aligned complete products. We prove that, in several cases, their packing chromatic numbers remain bounded by a constant. It is known that the packing coloring problem is NP-complete even for trees \cite{fiala}, which in turn implies its NP-completeness for block graphs. Since path-aligned complete product graphs constitute a subclass of block graphs, the results presented here identify a polynomial-time solvable special case within this otherwise NP-complete class. Moreover, Sloper \cite{sloper} proved that the packing chromatic number of any caterpillar is at most 7. In this work, we provide a complete structural characterization of caterpillars with packing chromatic number 3.

The remainder of this paper is organized as follows. Section \ref{ref:prelim} presents definitions and preliminary results. In Section \ref{sec:pathalignedproducts}, we first define path-aligned graph products and then study their packing chromatic number, focusing on path-aligned cycle products and path-aligned complete products. In particular, we determine the packing chromatic number of path-aligned cycle products for some specific number of vertices of the cycle lying on the path.
For complete graphs with three, four and five vertices, where at least two vertices lie on the path, we establish bounds on the packing chromatic number. Section \ref{sec:corona} turns to the corona of paths, where we provide a complete characterization of caterpillar graphs with packing chromatic number 3. 
Section \ref{sec:pathalignedproducts} and \ref{sec:corona} conclude with summarizing remarks and open research directions.

\section{Preliminaries} \label{ref:prelim}
Let $n$ be a positive integer. $P_n$, $C_n$, and $K_n$ denote a path, cycle, and complete graph on $n$ vertices, respectively. 
We will denote also the path \( P_n \) as $v_1v_2 \dots v_n$, i.e. the graph with
\( V(P_n) = \{v_1, \ldots, v_n\} \), where \( v_i v_{i+1} \in E(P_n) \) for every \( i \in \{1, \ldots, n-1\} \). 

%$K_{n,m}$ denote a complete bipartite graph with bipartitions of sizes $n$ and $m$.

\begin{prop}[\cite{goddard2008broadcast}]
    $\chi_{\rho}(P_n) =
\begin{cases} 
    2, & \text{for } 2 \leq n \leq 3, \\
    3, & \text{otherwise.}
\end{cases}$
\end{prop}

\begin{prop}[\cite{goddard2008broadcast}]
    $\chi_{\rho}(C_n) =
\begin{cases}
    3, & \text{for } n=3 \text{ or } $n$ \text{ is a multiple of } 4, \\
    4, & \text{otherwise.}
\end{cases}$
\label{prop:cycles}
\end{prop}

\begin{prop}[\cite{goddard2008broadcast}] \label{proo:goddard}
    Let $G$ be a graph of diameter $2$. Then $\chi_p(G)=n-\alpha(G)+1=\beta(G)+1$, where $\alpha(G)$ denotes the independence number of the graph $G$ while $\beta(G)$ denotes the size of its smallest vertex cover. 
\end{prop}

By Proposition \ref{proo:goddard}, it is clear that $\chi_p(K_{1,n})=2$, where $K_{1,n}$ denotes a star on $n+1$ vertices.

\begin{observation}
 Let $G$ be any arbitrary graph with $diam(G)=d$. Then any color greater than $d-1$ is used at most once in any packing coloring of $G$. \hfill $\Box$   \label{obs:diam}
\end{observation}

\begin{prop}
  Let $G$ be a graph with $\chi_p(G)\leq x$, where $x\geq 1$. Then no vertex of degree at least $x$ can be colored with $1$.
    \label{obs:no1}  
\end{prop}
\begin{proof}
    Suppose, to the contrary, that a vertex of degree at least $x$ is assigned color 1. All of its neighbors must receive distinct colors, since no two adjacent vertices can share the same color. Moreover, the neighbors of this vertex are at distance at most 2 from one another, so the only color that could potentially be reused among them is color 1. However, this is impossible because they are all adjacent to the vertex already colored with 1.
Therefore, at least 
$x+1$ distinct colors are required, contradiction.
\end{proof}

\section{Path-aligned graph products} \label{sec:pathalignedproducts}

An \emph{automorphism} of a graph is a permutation of its vertices such that edges are mapped to edges and non-edges are mapped to non-edges.

\begin{definition}[\cite{GoRo13}]
\emph{A graph $G$ is said to be} vertex-transitive \emph{if its automorphism group acts transitively on the vertex set of $G$, meaning that for any two vertices $u, v \in V(G)$, there exists an automorphism $\phi$ such that $\phi(u)=v$.}
\end{definition}

Informally speaking, a graph is vertex-transitive if every vertex has the same local environment, so that no vertex can be distinguished from any other based on the vertices and edges surrounding it. Examples of vertex-transitive graphs are: complete graphs, cycles, Petersen graph.
We introduce a graph product of a path and vertex-transitive graph.

\begin{definition}
\emph{Given positive integers $n$ and $\ell$, let us assume that $n$ is divisible by $\ell$. Let \( G \) be any connected vertex-transitive graph having $P_{\ell}$ as a subgraph.
%and let \( P_n \) be a path $P_n=(v_1, v_2, \dots, v_n)$.
% with \( V(P_n) = \{v_1, \ldots, v_n\} \), where \( v_i v_{i+1} \in E(P_n) \) for every \( i \in \{1, \ldots, n-1\} \). 
Then, the} path-aligned graph product of $G$\emph{, denoted by
\( P_n \Diamond_{\ell} G \), is the graph formed by the path \( P_n \) and \(  n/\ell  \) copies of \( G \), where the $i^{\text{th}}$ copy of $P_{\ell}$ in $P_n$, i.e. the vertices \( v_j, \ldots, v_{j+\ell-1} \) such that \( j = (i-1)\ell + 1 \),  corresponds to a $P_\ell$ in the $i^{\text{th}}$ copy of \( G \), where $1 \leq i\leq n/\ell$ (cf. Fig.\ref{fig:P10_C4}).} \label{aligned_cycles}
\end{definition}

\subsection{Path-aligned cycle products}\label{subsec:cycles}
    We start with paths of cycles having 2 or 3 common vertices with the path, i.e. $\ell \in \{2,3\}$. Inspired by \cite{laiche2017packing}, for $\ell=2$ we will use the following notation for the suggested coloring. For consecutive cycles connected by a path, we will define a pattern for the path and, at the same time, specify the pattern for the cycles. For example, the coloring of $P_{10} \Diamond_2 C_4$ shown in Figure~\ref{fig:P10_C4} can be described as: $({\bf 3\ }1\ 2\ {\bf 1})({\bf 4\ }1\ 2\ {\bf 1})({\bf 3\ }1\ 2\ {\bf 1})({\bf 5\ }1\ 2\ {\bf 1})({\bf 3\ }1\ 2\ {\bf 1})$, where bold numbers correspond to the colors of the vertices on the path $P_{10}$, while non-bold numbers represent colors of the vertices belonging only to the cycles. The colors for consecutive cycles are grouped using parentheses $()$. 
    %We mean that the colors without paranthesis are the colors for the consecutive vertices of the path while the numbers in paranthesis mean the colors for the consecutive vertices of the relevant cycle. According to this the coloring of $P_4 \Diamond_2 C_{4s+3}$, $s\geq 1$, based on that one given in Figure~\ref{fig:c4l+3} will be described as: $1(24[1213]^*121)31(42[1312]^*131)2$.
    %According to this, the coloring of $P_4 \Diamond_2 C_{11}$ based on that one from Figure \ref{fig:c4l+3} will be written as $({\bf 1}\ 2\ 4\ 1\ 2\ 1\ 3\ 1\ 2\ 1\ {\bf 3})({\bf 1}\ 4\ 2\ 1\ 3\ 1\ 2\ 1\ 3\ 1\ {\bf 2})$. 
    One can observe that the order of non-bold vertices follows the clockwise order of the vertices on the cycle, starting from the previous bold vertex.
For larger $\ell$ the notation needs to be slightly modified. For example, for the coloring of $P_{15} \Diamond_3 C_{4s}$, based on Figure \ref{fig:P15_C8}, it will be described as $({\bf 2}\ 1\ [3\ 1\ 2\ 1]^*\ \overleftarrow{{\bf 3\ 1}})$ $({\bf 1}\ 2\ [1\ 3\ 1\ 2]^* \overleftarrow{{\bf 1\ 4}})$ $({\bf 3}\ 1\ [2\ 1\ 3\ 1]^*\overleftarrow{{\bf 2\ 1}})$ $({\bf 1}\ 3\ [1\ 2\ 1\ 3]^*\overleftarrow{{\bf 1\ 4}})$ $({\bf 2}\ 1\ [3\ 1\ 2\ 1]^*\overleftarrow{{\bf 3\ 1}})$.
The left arrow above the colors denotes the colors of the corresponding vertices on the path in reverse order, ensuring that all colors of a single cycle appear consecutively within the parentheses. Thus, considering only the vertices of the path $P_{15}$, their colors are as follows: $2\ 1\ 3\ 1\ 4\ 1\ 3\ 1\ 2\ 1\ 4\ 1\ 2\ 1\ 3$. In addition, the notation $[\ ]^*$ indicates that the color pattern $[\ ]$ is repeated an arbitrary number of times, possibly zero. Similarly, $[\ ]^k$ denotes that the pattern $[\ ]$ is repeated exactly $k$ times.

\begin{prop}\label{prop:Cn}
    Let $t \ge 2$ and $n \ge 1$ be integers. 
    Then $P_{\ell t}\Diamond_{\ell} C_n$ is isomorphic to $P_{(n-\ell +2)t}\Diamond_{n-\ell+2} C_n$ for any integer $2 \leq \ell \leq n$.
\end{prop}
\begin{proof}
Let 
\( P_{\ell t} \Diamond_{\ell} C_n \) denote the graph formed by the path \( P_{\ell t} \) and \(  \ell  \) copies of \( C_n \), where the $i^{\text{th}}$ copy of $P_{\ell}$ in $P_{\ell t}$, i.e. the vertices \( v_j, \ldots, v_{j+\ell -1} \) such that \( j = (i-1)\ell + 1 \),  corresponds to a $P_{\ell}$ in the $i^{\text{th}}$ copy of \( G \), as given in Definition \ref{aligned_cycles}. Consider $C
_n$ as the union of paths $P_{\ell}$ and $P'_{n-\ell +2}$ with the identified start and end vertices. Then, let the vertices in the path $P'_{n-\ell +2}$ in the $i^{\text{th}}$ copy of $C_n$ in \( P_{\ell t} \Diamond_{\ell} C_n \) be \( u_s, \ldots, u_{s+(n-\ell+1)}\) such that $s=(i-1)(n-\ell+2)+1$, where $u_s$ and $u_{s+(n-\ell +1)}$ are identified with $v_j$ and $v_{j+\ell -1}$, respectively. Then, the path-aligned cycle product of the path consisting of the union of the paths $P'_{n-\ell +2}$ in each copy of $C_n$ is equivalent to the path-aligned cycle product $P_{(n-\ell +2)t}\Diamond_{n-\ell +2} C_n$. Hence, $P_{(n-\ell +2)t}\Diamond_{n-\ell +2} C_n$ is isomorphic to $P_{\ell t}\Diamond_{\ell} C_n$.
\end{proof}

In the following, we consider the packing coloring of path-aligned cycle products in accordance with the lengths of the cycles.
%\textcolor{orange}{What about such notation for the coloring of $P_6 \Diamond_3 C_{4s+3}$ from Figure \ref{fig:c4s+3_3}? Is this notation legible?}

%$$\overbrace{213}^{(4[1213]^*121)}\ \overbrace{151}^{4[2131]^*213}$$

%Alternative notation:
%$$\stackrel{(4[1213]^*121)}{213} \ \stackrel{4[2131]^*213}{151}$$
\subsubsection{Cycles $C_{4s}$ for $s \geq 1 $} 

%\paragraph{Two common vertices, i.e. $l=2$ (cf. Fig. \ref{fig:c4l})}
%%%%%%%%%%%%%%%%%%%%%%%
% P_{2t}\Diamond_2 C_{4s} %
%%%%%%%%%%%%%%%%%%%%%%
\begin{theorem}
    Let $s,t\geq 1$ be integers. Then,
$$\chi_p(P_{2t}\Diamond_2 C_{4s})\leq \left\{
\begin{array}{ll}
3, & t=1,\\
4, & 2 \leq t \leq 3,\\
5, & t\geq 4.
\end{array}\right.
$$
In addition, equality holds for $t\in \{1,2,3\}$. 
\label{thm:l=2C4s}   
\end{theorem}
\begin{proof}
For $t=1$, we have $P_{2t} \Diamond_2 C_{4s} \simeq C_{4s}$, and hence its packing chromatic number is equal to 3 (cf. Proposition \ref{prop:cycles}). 

\noindent For $t=2$, the coloring pattern is as follows: $({\bf 3}\ 1\ [2\ 1\ 3\ 1]^*2\ {\bf 1})({\bf 4}\ 1\ [2\ 1\ 3\ 1]^* 2\ {\bf 1}).$ 
\noindent For $t=3$, we have the following pattern: $$({\bf 3}\ 1\ [2\ 1\ 3\ 1]^*2\ {\bf 1})({\bf 4}\ 1\ [2\ 1\ 3\ 1]^* 2\ {\bf 1})({\bf 3}\ 1\ [2\ 1\ 3\ 1]^*2\ {\bf 1}).$$

\noindent In both cases, one can easily check that the distances between vertices colored with $i$ are at least $i+1$, $i\in \{1,2,3,4\}$. In addition, it is not possible to color $P_4 \Diamond_2 C_{4s}$ or $P_6 \Diamond_2 C_{4s}$ using fewer colors. In fact, it suffices to show that three colors are not sufficient. Indeed, it is impossible to obtain a packing coloring of $P_6$ without assigning color $1$ to the vertices of degree 3 in $P_6 \Diamond_2 C_{4s}$. This, in turn, already necessitates the use of at least four colors, by Proposition~\ref{obs:no1}. The path $P_4$ can be colored  using 3 colors while avoiding the assignment of color 1 to the vertices of degree 3 in $P_4 \Diamond_2 C_{4s}$, using the pattern $1\ 2\ 3\ 1$ or $1\ 3\ 2\ 1$. However, such a coloring cannot be extended to $C_{4s}$ without introducing a fourth color for any $s\geq 1$. Indeed, w.l.o.g. let us consider only pattern $1\ 2\ 3\ 1$ for $P_4$. Note that the vertex on the first cycle adjacent to $v_1$ (following the notation from Definition \ref{aligned_cycles}), distinct from $v_2$, cannot be assigned any of the colors 1, 2 or 3. Hence, color 4 must be used.
%Thus, the packing 4-colorings of $P_4 \Diamond_2 C_{4s}$ or $P_6 \Diamond_2 C_{4s}$ are optimal in the sense of the number of colors.

\noindent Let $(a)$, $(b)$, and $(c)$ denote the patterns $({\bf 3}\ 1\ [2\ 1\ 3\ 1]^*2\ {\bf 1})$, $({\bf 4}\ 1\ [2\ 1\ 3\ 1]^* 2\ {\bf 1})$, $({\bf 5}\ 1\ [2\ 1\ 3\ 1]^*2\ {\bf 1})$, respectively. For $t \geq 4$, if $t \equiv 0 \pmod{4}$, then we repeat the $(a)(b)(a)(c)$ pattern $t/4$ times. Otherwise, we repeat the $(a)(b)(a)(c)$ pattern $\lfloor t/4\rfloor$ times, and we add $(a)$, $(a)(b)$, or $(a)(b)(a)$ for $t\equiv 1, 2, 3 \pmod{4}$, respectively (cf. Figure \ref{fig:P10_C4} and \ref{fig:P10_C8}).
\\One can verify that the pattern ensures feasible packing 5-colorings.
\end{proof}

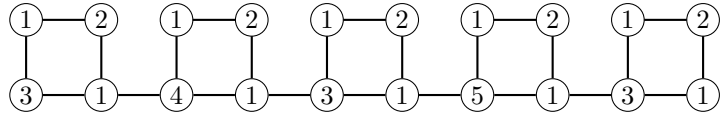
\begin{figure}[H]
\centering
\begin{tikzpicture}[scale=1.0, every node/.style={circle, draw=black, fill=white, inner sep=1.5pt}]

% Wierzchołki głównej ścieżki (10)
\foreach \i/\x/\col in {
  0/0/3, 1/1/1, 2/2/4, 3/3/1, 4/4/3,
  5/5/1, 6/6/5, 7/7/1, 8/8/3, 9/9/1
} {
  \node (v\i) at (\x,0) {\col};
}

% Krawędzie ścieżki głównej
\foreach \i in {0,...,8} {
  \pgfmathtruncatemacro{\j}{\i + 1}
  \draw[thick] (v\i) -- (v\j);
}

% Cykl C4 dla co drugiej krawędzi ścieżki (i parzyste)
\foreach \i/\x in {0/0, 2/2, 4/4, 6/6, 8/8} {
  \pgfmathtruncatemacro{\j}{\i + 1}
  \coordinate (a\i) at (\x,1);             % góra nad v_i
  \coordinate (b\i) at ({\x + 1},1);       % góra nad v_{i+1}
  
  % Rysuj cykl (prostokąt: v_i - a - b - v_{i+1} - v_i)
  \draw[thick] (v\i) -- (a\i) -- (b\i) -- (v\j);
  
  % Wierzchołki cyklu z kolorami
  \node at (a\i) {1};
  \node at (b\i) {2};
}

\end{tikzpicture}
\caption{5-packing coloring of $P_{10} \mathbin{\Diamond_2} C_4$.}\label{fig:P10_C4}
\end{figure}

\begin{figure}[htb]
\centering
\begin{tikzpicture}[scale=1.0, every node/.style={circle, draw=black, fill=white, inner sep=1.5pt}]

% Wierzchołki głównej ścieżki (10)
\foreach \i/\x/\col in {
  0/0/3, 1/1/1, 2/2/4, 3/3/1, 4/4/3,
  5/5/1, 6/6/5, 7/7/1, 8/8/3, 9/9/1
} {
  \node (v\i) at (\x,0) {\col};
}

% Krawędzie ścieżki głównej
\foreach \i in {0,...,8} {
  \pgfmathtruncatemacro{\j}{\i + 1}
  \draw[thick] (v\i) -- (v\j);
}

% Cykl C4 dla co drugiej krawędzi ścieżki (i parzyste)
\foreach \i/\x in {0/0, 2/2, 4/4, 6/6, 8/8} {
  \pgfmathtruncatemacro{\j}{\i + 1}
  \coordinate (a\i) at ({\x-0.2},0.7);             % góra nad v_i
  \coordinate (b\i) at ({\x-0.2},1.4);       % góra nad v_{i+1}
  \coordinate (c\i) at (\x,2.1);
  % Rysuj cykl (prostokąt: v_i - a - b - v_{i+1} - v_i)
  \coordinate (d\i) at ({\x+1.2},0.7);             % góra nad v_i
  \coordinate (e\i) at ({\x+1.2},1.4);       % góra nad v_{i+1}
  \coordinate (f\i) at ({\x+1},2.1);
  \draw[thick] (v\i) -- (a\i) -- (b\i) -- (c\i) -- (f\i) -- (e\i) -- (d\i) -- (v\j);
  
  % Wierzchołki cyklu z kolorami
  \node at (a\i) {1};
  \node at (b\i) {2};
  \node at (c\i) {1};
  \node at (f\i) {3};
  \node at (e\i) {1};
  \node at (d\i) {2};
}

\end{tikzpicture}
\caption{Packing 5-coloring of $P_{10} \mathbin{\Diamond_2} C_8$.}\label{fig:P10_C8}
\end{figure}
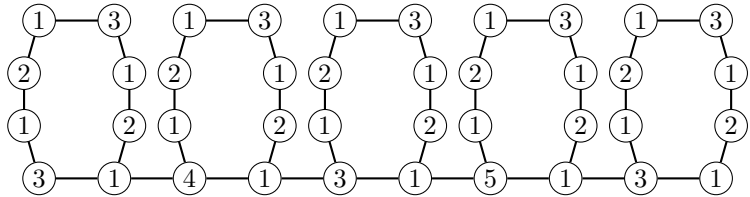

%%%%%%%%%%%%%%%%%%%%%%%
% P_{3t}\Diamond_3 C_{4s} %
%%%%%%%%%%%%%%%%%%%%%%
\begin{theorem}
    Let $s,t \geq 1$ be integers. Then,
$$\chi_p(P_{3t}\Diamond_3 C_{4s}) = \left\{
\begin{array}{ll}
3, & t=1,\\
4, & t\geq 2.
\end{array}\right.
$$
\end{theorem}
\begin{proof}
For $t=1$, we have $P_{3t} \Diamond_2 C_{4s} \simeq C_{4s}$ and hence its packing chromatic number is equal to 3 (cf. Proposition \ref{prop:cycles}).     

\noindent Let $(a)$, $(b)$, $(c)$, and $(d)$ denote the patterns $({\bf 2}\ 1\ [3\ 1\ 2\ 1]^* \overleftarrow{{\bf 3\ 1}})$, $({\bf 1}\ 2\ [1\ 3\ 1\ 2]^* \overleftarrow{{\bf 1\ 4}})$, $({\bf 3}\ 1\ [2\ 1\ 3\ 1]^* \overleftarrow{{\bf 2\ 1}})$, $({\bf 1}\ 3\ [1\ 2\ 1\ 3]^* \overleftarrow{{\bf 1\ 4}})$, respectively.
For $t \equiv 0 \pmod{4}$, we repeat the $(a)(b)(c)(d)$ pattern $t/4$ times. Otherwise, we repeat the $(a)(b)(c)(d)$ pattern $\lfloor t/4\rfloor $ times and add $(a)$ (cf. Fig.~\ref{fig:P15_C8}), $(a)(b)$, or $(a)(b)(c)$ for $t \equiv 1, 2, 3 \pmod{4}$, respectively.
One can easily verify that the pattern indicates a feasible packing 4-coloring.

\noindent Since it is not possible to avoid assigning color~$1$ to the vertices of degree~$3$ on the path in $P_{3t} \Diamond_3 C_{4s}$ for $t \geq 2$, Proposition~\ref{obs:no1} implies that fewer than four colors cannot be used.
\end{proof}

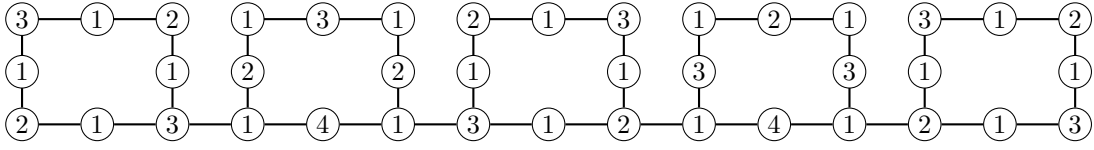
\begin{figure}[htb]
\centering
\begin{tikzpicture}[scale=1.0, every node/.style={circle, draw=black, fill=white, inner sep=1.5pt}]

% Wierzchołki głównej ścieżki (10)
\foreach \i/\x/\col in {
  0/0/2, 1/1/1, 2/2/3, 3/3/1, 4/4/4,
  5/5/1, 6/6/3, 7/7/1, 8/8/2, 9/9/1, 10/10/4, 11/11/1, 12/12/2, 13/13/1, 14/14/3
} {
  \node (v\i) at (\x,0) {\col};
}

% Krawędzie ścieżki głównej
\foreach \i in {0,...,13} {
  \pgfmathtruncatemacro{\j}{\i + 1}
  \draw[thick] (v\i) -- (v\j);
}

% Cykl C4 dla co drugiej krawędzi ścieżki (i parzyste)
\foreach \i/\x/\labA/\labB/\labC/\labE/\labF in {
  0/0/1/3/1/2/1,
  3/3/2/1/3/1/2,
  6/6/1/2/1/3/1,
  9/9/3/1/2/1/3,
  12/12/1/3/1/2/1
} {
  \pgfmathtruncatemacro{\j}{\i + 2}
   \pgfmathtruncatemacro{\s}{\i + 1}
  \coordinate (a\i) at ({\x},0.7);             % góra nad v_i
  \coordinate (b\i) at (\x,1.4);       % góra nad v_{i+1}
  \coordinate (c\i) at ({\s},1.4);
  % Rysuj cykl (prostokąt: v_i - a - b - v_{i+1} - v_i)
  %\coordinate (d\i) at ({\x+1.2},0.7);             % góra nad v_i
  \coordinate (e\i) at ({\j},1.4);       % góra nad v_{i+1}
  \coordinate (f\i) at ({\j},0.7);
  \draw[thick] (v\i) -- (a\i) -- (b\i) -- (c\i) --  (e\i) -- (f\i) -- (v\j);
\node at (a\i) {\labA};
  \node at (b\i) {\labB};
  \node at (c\i) {\labC};
  \node at (e\i) {\labE};
  \node at (f\i) {\labF};
}

\end{tikzpicture}
\caption{Packing 4-coloring of $P_{15} \mathbin{\Diamond_3} C_8$.}
    \label{fig:P15_C8}
\end{figure}

\subsubsection{Cycles $C_{4s+1}$ for $s\geq 1$} 

%%%%%%%%%%%%%%%%%%%%%%%
% P_{2t}\Diamond_2 C_{5} %
%%%%%%%%%%%%%%%%%%%%%%
\begin{theorem} Let $t \geq 1$ be an integer. Then,
      $$\chi_p(P_{2t} \Diamond_2 C_5) \leq \left \{
    \begin{array}{ll}
4, & t=1,\\
5, & 2 \leq t\leq 4,\\
6, & t \geq 5.
\end{array}\right.
    $$
In addition, equality holds for $t\in \{1,2,3,4\}$.
\end{theorem}
\begin{proof}
For $t=1$, we have $P_{2t}\Diamond_2 C_5 \simeq C_5$; therefore, by Proposition \ref{prop:cycles}, its packing chromatic number is equal to 4.

\noindent Let $(a)$, $(b)$, $(c)$, and $(d)$ denote the patterns $({\bf 1}\ 3\ 5\ 1\ {\bf 2})$, $({\bf 1}\ 3\ 1\ 2\ {\bf 4})$, $({\bf 1}\ 2\ 3\ 1\ {\bf 5})$, $({\bf 1}\ 3\ 1\ 4\ {\bf 2})$, respectively. 

\noindent For $2 \leq t \leq 4$, we use the patterns $(a)(b)$, $(a)(b)(c)$, and $(a)(b)(c)(d)$, respectively. One can observe that a packing coloring of $P_{4} \Diamond_2 C_{5}$ requires at least five colors. Indeed, if any vertex of degree~$3$ in $P_4 \Diamond_2 C_5$ is assigned color~$1$, its three neighbors must be assigned colors~$2$, $3$, and~$4$. However, color 4 cannot be reused, as there is no other vertex at distance at least 5); consequently, one of the two cycles $C_5$ cannot obtain color 4.  Since $\chi_p(C_5)=4$, color 5 must be used in such a case. On the other hand, let us omit assigning color 1 to vertices of degree 3 in $P_4 \Diamond_2 C_5$. W.l.o.g. let us color the vertices with 2 and 3. However, then color 3 cannot be reused, because there is no other vertex at distance at least 4 in $P_4 \Diamond_2 C_5$. Again, colors 4 and 5 must be used. Hence, we have $\chi_p(P_{2t} \Diamond_2 C_5)=5$ for $2\leq t \leq 4$.    

\noindent Now, let $(a')$, $(b')$, and $(c')$, denote the patterns $({\bf 1}\ 3\ 4\ 1\ {\bf 2})$, $({\bf 1}\ 3\ 4\ 2\ {\bf 5})$, $({\bf 1}\ 3\ 4\ 2\ {\bf 6})$, respectively. For $t \geq 5$, we use the pattern $(a')(b')(a')(c')$ for $\lfloor t/4\rfloor $ times and add $(a')$, $(a')(b')$, or $(a')(b')(a')$ for $i = 1, 2, 3$, respectively, for $t = 4k + i$ for some integer $k$. One can see that this results in a feasible packing $6$-coloring of $(P_{2t} \Diamond_2 C_5)$ for $t \geq 5$.

\end{proof}
%(cf. Fig.~\ref{fig:p8_c5}, \ref{fig:c5})
%\begin{figure}
 %   \centering
  %  \includegraphics[scale=0.1]{P8_C5.jpeg}
   % \caption{}
    %\label{fig:p8_c5}
%\end{figure}

%\begin{figure}
 %   \centering
  %  \includegraphics[scale=0.1]{C5.jpeg}
   % \caption{}
    %\label{fig:c5}
%\end{figure}

%%%%%%%%%%%%%%%%%%%%%%%
% P_{2t}\Diamond_2 C_{4s+1}, s\geq 2 %
%%%%%%%%%%%%%%%%%%%%%%

\begin{theorem}
    Let $t\geq 1$ and $s\geq 2$ be integers. Then,
    $$\chi_p(P_{2t}\Diamond_2 C_{4s+1})\leq\left\{
\begin{array}{ll}
4, & 1\leq t \leq 3,\\
5, & t \geq 4.
\end{array}\right.
$$
In addition, equality holds for $t\in\{1,2,3\}$. 
%(cf. \ref{fig:c4l+1})
\end{theorem}
\begin{proof}
For $t=1$, we have $P_{2t}\Diamond_2 C_{4s+1} \simeq C_{4s+1}$; hence, by Proposition \ref{prop:cycles}, its packing chromatic number is equal to 4. In addition, any packing coloring of $P_{2t}\Diamond_2 C_{4s+1}$ requires at least 4 colors, for any $t$. 

\noindent Now, let $(a)$, $(b)$, and $(c)$, denote the patterns
$({\bf 1}\ 3\ 2\ 1\ 4[1\ 3\ 1\ 2]^*1\ 3\ 1\ {\bf 2})$, $({\bf 1}\ 3\ 1\ 2\ 1\ [3\ 1\ 2\ 1]^*3\ 2\ 1\ {\bf 4})$, and $({\bf 1}\ 3\ 1\ 2\ 1\ [3\ 1\ 2\ 1]^*3\ 2\ 1\ {\bf 5})$, respectively.
For $2 \leq t \leq 4$, we use the patterns $(a)(b)$, $(a)(b)(a)$, and $(a)(b)(a)(c)$, respectively. For $t \in \{2, 3\}$, this yields an optimal packing 4-coloring.

\noindent For $t\geq 5$, we repeat the $(a)(b)(a)(c)$ pattern $\lfloor t/4\rfloor $ times and add $(a)$, $(a)(b)$, or $(a)(b)(a)$ for $i = 1, 2, 3$, respectively for $t = 4k + i$ for some integer $k$. In all cases, one can verify that the patterns ensure feasible packing 5-coloring. 
\end{proof}

%\begin{figure}
 %   \centering
  %  \includegraphics[scale=0.2]{C4l+1.jpeg}
   % \caption{A pattern for cycles of length $=4s+1$, $s\geq 2$. For longer path we repeat the pattern.}
    %\label{fig:c4l+1}
%\end{figure}

%In the following we will try to justify that for a packing coloring of $P_{2t} \Diamond_2 C_5$ we really need to use at least 6 colors for long enough paths. Let $v_1,\ldots,v_{2t}$ denote vertices of the path $P_{2t}$ in $P_{2t} \Diamond_2 C_5$. We assume that $t\geq 5$.

%\begin{observation}
  %  It is not possible that color 1 is not used to color vertices on the path in $P_{2t} \Diamond_2 C_5$, $t \geq 5$, in any packing coloring of $P_{2t} \Diamond_2 C_5$ with 5 colors.
%\end{observation}
%\begin{observation}
   % Color 3 cannot be assigned to vertices $v_i$ where $i \in \{3,4,\ldots,2t-2\}$, i.e. color 3 cannot apear on the path in interior cycles.
%\end{observation}
%\begin{observation}
  %  Color 5 cannot be assigned to vertices $v_i$ where $i \in \{3,4,\ldots,2t-2$, i.e. color 5 cannot apear on the path in interior cycles.
%\end{observation}

%Taking the above into account we can use only colors 1, 2, and 4 for vertices $v_3,\ldots,v_{2t-2}$, but for $t\geq 6$ it is not possible. Thus, we have 

%$$\chi_p(P_{2t}\Diamond_2 C_5)=6$$
%for $t\geq 6.$ Thus, at least one graph, i.e. $P_{10}\Diamond_2 C_5$ can be colored with 5 not 6 colors.

%%%%%%%%%%%%%%%%%%%%%%%
% P_{3t}\Diamond_3 C_{5} %
%%%%%%%%%%%%%%%%%%%%%%
\begin{theorem} Let $t\geq 1$ be an integer. Then,
 $$\chi_p(P_{3t}\Diamond_3 C_5)=\left\{
\begin{array}{ll}
4, & t=1,\\
5, & t \geq 2.
\end{array}\right.
$$
\label{thm:l3C5}
\end{theorem}
\begin{proof}
For $t=1$, we have $P_{3t}\Diamond_3 C_{5} \simeq C_5$. Therefore, by Proposition \ref{prop:cycles}, its packing chromatic number is equal to 4. For $t=2$, we use the following coloring pattern: $({\bf 1}\ 3\ 2\ \overleftarrow{{\bf 1\ 5}})({\bf 4}\ 3\ 1\ \overleftarrow{{\bf 2\ 1}})$. For $t \geq 3$, we repeat this pattern as many times as necessary; that is,  $\lfloor t/2 \rfloor$ times, with the initial part of the pattern added if required. One can observe that even in the case of the packing coloring of $P_6 \Diamond_3 C_5$, fewer than five colors cannot be used. Indeed, if any vertex of degree 3 in $P_6 \Diamond_3 C_5$ is assigned color 1, its three neighbors must receive colors $2,3$ and $4$. However, then color 4 cannot be repeated, as there is no other vertex at distance at least 5, that is, one of the two cycles $C_5$ cannot be assigned color 4.  Since $\chi_p(C_5)=4$, color 5 must be used in such a case. On the other hand, suppose we avoid assigning color~$1$ to the vertices of degree~$3$ in $P_6 \Diamond_3 C_5$. Note that at least one of these vertices must then be assigned color~$3$ or~$4$. The color used in this case cannot be reused, since there is no other vertex at distance at least 4 in $P_6 \Diamond_3 C_5$. Consequently, color 5 must be used. Therefore, $\chi_p(P_{3t} \Diamond_3 C_5)=5$ for $t \geq 2$.
\end{proof}
%(cf. \ref{fig:c4s+1_3})
%%%%%%%%%%%%%%%%%%%%%%%
% P_{3t}\Diamond_3 C_{4s+1} %
%%%%%%%%%%%%%%%%%%%%%%
\begin{theorem}
    Let $t \geq 1$ and $s\geq 2$ be integers. Then, 
  $$\chi_p(P_{3t}\Diamond_3 C_{4s+1})\leq\left\{
\begin{array}{ll}
4, & 1 \leq t \leq 2,\\
5, & t \geq 3,
\end{array}\right.
$$  
In addition, equality holds for $t \in \{1,2\}$.
\label{thm:l3C4s+1}
\end{theorem}
\begin{proof}
For $t=1$, we have $P_{3t}\Diamond_3 C_{4s+1} \simeq C_{4s+1}$. Hence, by Proposition \ref{prop:cycles}, its packing chromatic number is equal to 4. 
\\For $t=2$, the coloring pattern is as follows: $({\textbf 3}\ 1\ 4\ 2\ 1\ 3\ 1\ [2\ 1\ 3\ 1]^*\overleftarrow{{\bf 2\ 1}})$ $({\textbf 1}\ 3\ 1\ 2\ 1\ 3\ [1\ 2\ 1\ 3]^*1\ \overleftarrow{{\bf 2\ 4}})$.
%(cf. Fig.~\ref{fig:P6l3C9}). 
One can verify that this leads to an optimal packing 4-coloring.  

\noindent For $t \geq 3$, let $(a)$, $(b)$, $(c)$, and $(d)$ denote the patterns $({\bf 1}\ 4\ 2\ 1\ 3\ 1\ [2\ 1\ 3\ 1]^*2\ \overleftarrow{{\bf 1\ 5}})$, $({\bf 3}\ 1\ 2\ 1\ 3\ 1\ [2\ 1\ 3\ 1]^*4\ \overleftarrow{{\bf 2\ 1}})$, $({\bf 1}\ 3\ 1\ 2\ 1\ 3\ [1\ 2\ 1\ 3]^*4\ \overleftarrow{{\bf 1\ 5}})$, $({\bf 2}\ 1\ 3\ 4\ 1\ 2\ [1\ 3\ 1\ 2]^*1\ \overleftarrow{{\bf 3\ 1}})$, respectively. We repeat the $(a)(b)(c)(d)$ pattern $\lfloor t/4 \rfloor$ times and add $(a)$, $(a)(b)$, or $(a)(b)(c)$ for $i = 1, 2, 3$, respectively, for $t = 4k + i$ and some integer $k$.

\noindent One can verify that this pattern leads to a feasible packing 5-coloring.

\end{proof}
%One can observe that the last pattern does not lead to feasible packing 5-coloring of $P_{3t}\Diamond_3 C_5$.
%\begin{figure}
 %   \centering
    %\includegraphics[width=0.7\linewidth]{c4s+1_3.jpeg}
    %\caption{$P_{3t}\Diamond_3 C_{4s+1}$, $s\geq 2$, and $P_{3t}\Diamond_3 C_5$ for $t \geq 5$.}
    %\label{fig:c4s+1_3}
%\end{figure}
%\begin{figure}
 %   \centering
  %  \includegraphics[width=0.5\linewidth]{P6l3C9.jpeg}
   % \caption{Packing 4-coloring of $P_6 \Diamond_3 C_{4s+1}$.}
    %\label{fig:P6l3C9}
%\end{figure}

\subsubsection{Cycles $C_{4s+2}$ for $s\geq 1$}
%%%%%%%%%%%%%%%%%%%%%%%%%%%%%%%%%%%%
% P_{2t}\Diamond_2 C_{4s+2}
%%%%%%%%%%%%%%%%%%%%%%%%%%%%%%%%%%%%
\begin{theorem}
    Let $t\geq 1$ and $s\geq 1$ be integers. Then,
  $$\chi_p(P_{2t}\Diamond_2 C_{4s+2})\leq\left\{
\begin{array}{ll}
4, & 1 \leq t \leq 2,\\
5, & t \geq 3.
\end{array}\right.
$$ 
In addition, equality holds for $t \in \{1,2\}$. 
%(cf. Fig. \ref{fig:c4l+2})
\end{theorem}
\begin{proof}
For $t=1$ we have $P_{2t}\Diamond_2 C_{4s+2} \simeq C_{4s+2}$. Therefore, by Proposition \ref{prop:cycles}, its packing chromatic number is equal to 4.  
\\For $t=2$, the coloring pattern is as follows: $$({\bf 1}\; 2\; 4\; [1\;2\;1\;3]^* 1\; 2\; {\bf 3})({\bf 1}\; 4\; 1\; [3\; 1\; 2\; 1]^*3\; 1\; {\bf 2}).$$
It can be easily verified that the pattern results in a proper packing coloring. Moreover, since $C_{4s+2}$ is a subgraph of this graph and requires four colors, the 4-packing coloring of $P_{2t}\Diamond_2 C_{4s+2}$ obtained according to the given pattern is optimal.
\\ For $t\geq 3$, we use another coloring pattern: $$({\bf 1}\; 5\; 4\; [1\;2\;1\;3]^* 1\; 2\; {\bf 3})({\bf 1}\; 4\; 1\; [3\; 1\; 2\; 1]^*3\; 1\; {\bf 2}).$$ We repeat this pattern $\lfloor t/2 \rfloor$ times and apply the relevant initial part of this pattern for the remaining uncolored part of $P_{2t} \Diamond_2 C_{4s+2}$, if needed.

One can verify that all the patterns ensure a feasible 4- or 5-packing coloring.
\end{proof}

Note that the pattern corresponding to the optimal packing $3$-coloring of the path $P_{2t}$ is utilized to construct the packing $5$-coloring of $P_{2t} \Diamond_2 C_{4s+2}$.

%(cf. Fig.~\ref{fig:c4l+2}). 
%\begin{figure}
 %   \centering
 %   \includegraphics[scale=0.2, angle=90]{C4l+2.jpeg}
  %  \caption{$C_{4s+2}$}
   % \label{fig:c4l+2}
%\end{figure}

%%%%%%%%%%%%%%%%%%%%%%%%%%%%%%%%%%%%
% P_{3t}\Diamond_3 C_{4s+2}
%%%%%%%%%%%%%%%%%%%%%%%%%%%%%%%%%%%%
\begin{theorem} Let $t \geq 1$ be an integer. Then,
    $$\chi_p(P_{3t}\Diamond_3 C_6)\leq\left\{
\begin{array}{ll}
4, & 1 \leq t \leq 2,\\
5, & t \geq 3.
\end{array}\right.
$$
In addition, the equality holds for $t \in \{1,2\}$. 
\end{theorem}
\begin{proof}
For $t=1$, we have $P_{3t} \Diamond_3 C_6 \simeq C_6$, and by Proposition \ref{prop:cycles}, its packing chromatic number is equal to 4.
\\For $t=2$, the coloring pattern is as follows:
$$({\bf 1}\; 3\; 1\; 4\; \overleftarrow{{\bf 1\;2}})({\bf 3}\; 2\; 1\; 4\; \overleftarrow{{\bf 2\;1}}).$$
It can be easily verified that the pattern ensures a proper packing coloring.
Moreover, since $C_6$ is a subgraph of $P_{3t} \Diamond_3 C_6$ and $\chi_p(C_6)=4$ by Proposition \ref{prop:cycles}, the 4-packing coloring of the graph obtained according to the given pattern is optimal.
\\For $t\geq 3$, we use another coloring pattern. %(cf. Fig.~\ref{fig:c4s+2_w3}). 
Let $(a)$, $(b)$, $(c)$, and $(d)$ denote the patterns
$({\bf 1}\;5\;1\;4\;\overleftarrow{{\bf 1\;3}})$, $({\bf 2}\;1\;5\;1\;\overleftarrow{{\bf 3\; 1}})$, $({\bf 1}\; 4\;1\;5\;\overleftarrow{{\bf 1\;2}})$, and $({\bf 3}\;1\;4\;1\;\overleftarrow{{\bf 2\; 1}})$ respectively. 
For $t \equiv 0 \pmod{4}$, we repeat the $(a)(b)(c)(d)$ pattern $t/4$ times. Otherwise, we repeat the $(a)(b)(c)(d)$ pattern $\lfloor t /4 \rfloor$ times and add $(a)$, $(a)(b)$ or $(a)(b)(c)$ for $t \equiv 1, 2, 3 \pmod{4}$, respectively.  
One can verify that the pattern results in a feasible 5-packing coloring.
\end{proof}
\begin{theorem}
    Let $t\geq 1$ and $s\geq 2$ be integers. Then, 
    $$\chi_p(P_{3t}\Diamond_3 C_{4s+2})\leq\left\{
\begin{array}{ll}
4, & 1 \leq t \leq 3,\\
5, & t \geq 4.
\end{array}\right.
$$
In addition, the equality holds for $t \in \{1,2,3\}$. 
%(cf. Fig.~\ref{fig:c4s+2_w3})
\end{theorem}
\begin{proof}
    For $t = 1$, we have $P_{3t}\Diamond_{3} C_{4s+2} \simeq C_{4s+2}$, and by Proposition \ref{prop:cycles}, its packing chromatic number is equal to 4.
\\For $t\in\{2,3\}$, we use the relevant part of the pattern:
%(cf. Fig.~\ref{fig:c4s+2_3}):
$$({\bf 1}\;3\;1\;2\;1\;3\;[1\;2\;1\;3]^*1\;4\;\overleftarrow{{\bf 1\; 2}})({\bf 3}\; 2\;1\;4\;1\;2\;3\;1\;[2\;1\;3\;1]^*\overleftarrow{{\bf 2\;1}})({\bf 1}\; 4\;1\;2\;1\;3\;[1\;2\;1\;3]^*1\;2\;\overleftarrow{{\bf 1\;3}}).$$
It can be easily verified that the pattern ensures a feasible packing coloring.
Moreover, since $C_{4s+2}$ is a subgraph of $P_{3t}\Diamond_{3} C_{4s+2}$ and $\chi_p(C_{4s+2})=4$, the 4-packing coloring of the graph $P_{3t}\Diamond_3 C_{4s+2}$, where $t \leq 3$, obtained
according to the given pattern is optimal.
\\For $t\geq 4$, we use another coloring pattern. %(cf. Fig. \ref{fig:c4s+2_w3}). 
Let $(a)$, $(b)$, $(c)$, and $(d)$ denote the patterns
$({\bf 1}\;5\;1\;3\;1\;2\;[1\;3\;1\;2]^*
1\;4\;\overleftarrow{{\bf 1\;3}})$, $({\bf 2}\;1\;3\;1\;2\;1\;[3\;1\;2\;1]^*
5\;1\;\overleftarrow{{\bf 3\; 1}})$, $({\bf 1}\; 4\;1\;2\;1\;3\;[1\;2\;1\;3]^*
1\;5\;\overleftarrow{{\bf 1\;2}})$, and $({\bf 3}\;1\;2\;1\;3\;[1\;2\;1\;3]^*
1\;4\;1\;\overleftarrow{{\bf 2\; 1}})$, respectively. 

If $t \equiv 0 \pmod{4}$, we repeat the $(a)(b)(c)(d)$ pattern $t/4$ times. Otherwise, we repeat the $(a)(b)(c)(d)$ pattern $\lfloor t /4 \rfloor$ times and add $(a)$, $(a)(b)$ or $(a)(b)(c)$ for $t \equiv 1, 2, 3 \pmod{4}$, respectively.

One can verify that the pattern results in a feasible 5-packing coloring.
\end{proof}
\subsubsection{Cycles $C_{4s+3}$ for $s\geq 0$ }
%%%%%%%%%%%%%%%%%%%%%%%%%%%%%%%%%%%%
% P_{2t}\Diamond_2 C_{4s+3}
%%%%%%%%%%%%%%%%%%%%%%%%%%%%%%%%%%%%
\begin{theorem} \label{thm:pathAlignedCycle}
Let $t\geq 1$ be an integer. Then, 
    $$\chi_p(P_{2t}\Diamond_2 C_3)\leq\left\{
\begin{array}{ll}
3, & t=1,\\
4, & 2 \leq t \leq 3,\\
5, & t\geq 4.
\end{array}\right.
$$
In addition, equality holds for $t\in\{1,2,3\}$. 
%(cf. Fig.~\ref{fig:c3})
\end{theorem}
\begin{proof}
The case $t = 1$ is straightforward, as each vertex is assigned a distinct color. For $t = 2$, we obtain two cycles $C_3$, and since no two vertices are at distance at least~$4$, color~$3$ can be used only once. Consequently, a fourth color is required, contradiction.  Now, let $(a)$, $(b)$, and $(c)$ denote the patterns 
$({\bf 3}\;2\;{\bf 1})$, $({\bf 4}\; 2\; {\bf 1})$, $({\bf 5}\;2\;{\bf 1})$, respectively. For $t=2$, we use the pattern $(a) (b)$, while for $t=3$, this pattern can be extended to $(a)(b)(a)$.

%$({\bf 3}\;2\;{\bf 1})({\bf 4}\; 2\; {\bf 1})$. This coloring can be easily extended for $t=3$: $({\bf 3}\;2\;{\bf 1})({\bf 4}\; 2\; {\bf 1})({\bf 3}\;2\;{\bf 1})$. \textcolor{red}{Up to now, the colorings consume the smallest number of colors.}

\noindent In the case of $t\geq 4$, if $t \equiv 0 \pmod{4}$, we repeat the $(a)(b)(a)(c)$ pattern $t/4$ times. Otherwise, we repeat the $(a)(b)(a)(c)$ pattern $\lfloor t /4 \rfloor$ times and add $(a)$, $(a)(b)$ or $(a)(b)(a)$ for $t \equiv 1, 2, 3 \pmod{4}$, respectively.

One can verify that the pattern results in a feasible 5-packing coloring.
\end{proof}
\begin{observation}
    The pattern $[1\ 2\ 1\ 3]$ for 3-packing coloring of $P_{2t}$ cannot be extended to a 5-packing coloring of $P_{2t} \Diamond_2 C_3$, $t\geq 4$.
\end{observation}

Indeed, if the path is colored with the pattern $[1\ 2\ 1\ 3]$, the remaining uncolored vertex in each cycle must be assigned color~$4$ or~$5$, since colors~$2$ and~$3$ cannot be used, as the distances to the vertices on the path already colored with~$2$ or~$3$ are insufficient. Color~$4$ can be assigned to every second cycle; however, the distance between the uncolored vertices in consecutive second cycles is insufficient for the reuse of color~$5$. Consequently, color~$6$ must be used.

%\begin{figure}
 %   \centering
  %  \includegraphics[scale=0.1]{c3.jpeg}
   %     \caption{$C_3$}
%    \label{fig:c3}
%\end{figure}

\begin{theorem}
    Let $s\geq 1$ and $t\geq 1$ be integers. Then,
    $$\chi_p(P_{2t}\Diamond_2 C_{4s+3})\leq\left\{
\begin{array}{ll}
4, & 1 \leq t \leq 2,\\
5, & t\geq 3.
\end{array}\right.
$$
In addition, equality holds for $t\in\{1,2\}$.
\end{theorem}
\begin{proof}
    It is clear that fewer than 4 colors cannot be used. For $t=1$ the graph is isomorphic to $C_{4s+3}$, and by Proposition \ref{prop:cycles} 4 colors are necessary. 
    \\For $t=2$ we color the graph using the pattern:
    $$({\bf 1}\; 2\;4\;[1\;2\;1\;3]^*\;1\;2\;1\;{\bf 3})({\bf 1}\; 4\; 2\;[1\;3\;1\;2]^*\;1\;3\;1\;{\bf 2}).$$
    \\For $t\geq 3$ we apply a slightly different pattern:
    $$({\bf 1}\; 5\;4\;[1\;2\;1\;3]^*\;1\;2\;1\;{\bf 3})({\bf 1}\; 4\; 2\;[1\;3\;1\;2]^*\;1\;3\;1\;{\bf 2}).$$
    \\We repeat this pattern $\lfloor t/2 \rfloor$ times and apply the relevant initial part of this pattern for the remaining uncolored part of $P_{2t}\Diamond_2 C_{4s+3}$, if needed.

    One can easily verify that all the patterns result in feasible 4-packing or 5-packing colorings. 
    %(Fig. \ref{fig:c4l+3})
\end{proof}

%\begin{figure}[htb]
 %   \centering
  %  \includegraphics[scale=0.2]{c4l-1.jpeg}
   % \caption{$P_{2t}\Diamond_2 C_{4s+3}$}
    %  \label{fig:c4l+3}
%\end{figure}

%%%%%%%%%%%%%%%%%%%%%%%%%%%%%%%%%%%%
% P_{3t}\Diamond_3 C_{4s+3}
%%%%%%%%%%%%%%%%%%%%%%%%%%%%%%%%%%%%
\begin{theorem}Let $t\geq 1$ be an integer. Then,
    $$\chi_p(P_{3t}\Diamond_3 C_3)\leq\left\{
\begin{array}{ll}
3, & t =1,\\
4, & 2 \leq t\leq 3,\\
5, & t \geq 4.
\end{array}\right.
$$
In addition, equality holds for $t\in\{1,2,3\}.$ 
\end{theorem}
\begin{proof}
    Clearly, $P_{3}\Diamond_3 C_3 \simeq C_3$. Thus, for $t=1$, we have $\chi_p(P_{3}\Diamond_3 C_3)=3$. For $t\geq 2$ at least 4 colors are needed. Indeed, if three colors were sufficient, color~$3$ would have to be used twice in a packing coloring of $P_6 \Diamond_3 C_3$. One can observe that any 3-packing coloring of one cycle $C_3$ in $P_6 \Diamond_3 C_3$ cannot be extended to a 3-packing coloring of the entire $P_6 \Diamond_3 C_3$. Therefore, 4 colors must be used. The relevant 4-packingcoloring of $P_{3t}\Diamond_3 C_3$ for $t \in \{2,3\}$ can be obtained using the following coloring pattern:

    $$({\bf 3}\; \overleftarrow{{\bf 1\;2}})({\bf 4}\; \overleftarrow{{\bf 1\;2}})({\bf 3}\; \overleftarrow{{\bf 1\;2}}).$$
\\    For $t\geq 4$, a closely related coloring pattern is applied: 
    $$({\bf 3}\; \overleftarrow{{\bf 1\;2}})({\bf 4}\; \overleftarrow{{\bf 1\;2}})({\bf 3}\; \overleftarrow{{\bf 1\;2}})({\bf 5}\; \overleftarrow{{\bf 1\;2}}).$$
    We repeat this pattern $\lfloor t/4 \rfloor$ times and, if necessary, the corresponding initial segment of the pattern is applied to the remaining uncolored portion of $P_{3t} \Diamond_3 C_3$.

    One can easily verify that all patterns result in feasible 4-packing or 5-packing colorings. %(cf. Fig.~\ref{fig:c3_3})
\end{proof}
%\begin{figure}[htb]
 %   \centering
  %  \includegraphics[width=0.5\linewidth]{c3_3.jpeg}
   % \caption{$P_{3t} \Diamond_3 C_3$.}
   % \label{fig:c3_3}
%\end{figure}

\begin{theorem}
    Let $s\geq 1$ and $t \geq 1$ be integers. Then,
 $$\chi_p(P_{3t}\Diamond_3 C_{4s+3})\leq\left\{
\begin{array}{ll}
4, & 1 \leq t \leq 2,\\
5, & t\geq 3.
\end{array}\right.
$$  
In addition, equality holds for $t\in\{1,2\}$. 
\end{theorem}
\begin{proof}
    By Proposition \ref{prop:cycles} and since $P_{3}\Diamond_3 C_{4s+3} \simeq C_{4s+3}$, $P_{3}\Diamond_3 C_{4s+3}$ requires 4 colors.
    \\For $t=1$ or $t=2$, we apply the relevant part of the following pattern:

    $$({\bf 2}\; 4\; [1\;2\;1\;3]^*\; 1\; 2\; 1\; \overleftarrow{{\bf 3\; 1}})({\bf 1}\; 4\; [2\;1\;3\;1]^*\;2\;1\;3\; \overleftarrow{{\bf 1\; 2}}).$$
  The resulting packing coloring uses four colors, which is the optimal number.
    \\For $t\geq 3$, we apply the following:
    $$({\bf 2}\; 4\; [1\;2\;1\;3]^*\; 1\; 2\; 1\; \overleftarrow{{\bf 3\; 1}})({\bf 1}\; 4\; [2\;1\;3\;1]^*\;2\;1\;3\; \overleftarrow{{\bf 1\; 5}}).$$
    We repeat this pattern $\lfloor t/2 \rfloor$ times and apply the initial part of the pattern for the remaining uncolored part of $P_{3t}\Diamond_3 C_{4s+3}$, if needed.

 It can be readily verified that the pattern yields a valid packing $5$-coloring. Indeed, the vertices assigned color~$i$ are at pairwise distances of at least~$i+1$, for $1 \leq i \leq 5$. %(cf. Fig.~\ref{fig:c4s+3_3}).
 
\end{proof}

%\begin{figure}
 %   \centering
  %  \includegraphics[width=0.5\linewidth, angle=90]{c4s+3_3.jpeg}
   % \caption{$P_{3t} \Diamond_3 C_{4s+3}$.}
    %\label{fig:c4s+3_3}
%\end{figure}

\subsubsection{Summarizing remarks}

All results established in Subsection \ref{subsec:cycles} are summarized in Tables \ref{tabela2} and \ref{tabela3}, and can be restated as Corollary \ref{ref:SummaryCorollary}.

\begin{table}[htb]
\begin{center}
\caption{The exact values and bounds for $\chi_p(P_{\ell t} \Diamond_{\ell} G)$, where $G$ is a cycle.}\label{tabela2}

\vspace{3mm}
\setlength{\tabcolsep}{12pt}
 {\tabulinesep=1mm
\begin{tabu}{|c|*{3}{c|}}\hline
\backslashbox[20mm]{$\ell$}{$G$} & $C_{4s}$ & $C_5$ & $C_{4s+1}$ \\ \hline

$2$ & $\left\{
\begin{array}{ll}
= 3, & t=1\\
= 4, & 2 \leq t \leq 3\\
\leq 5, & t\geq 4
\end{array}\right.$ & $\left \{
    \begin{array}{ll}
= 4, & t=1\\
= 5, & 2 \leq t\leq 4\\
\leq 6, & t \geq 5
\end{array}\right.$ & $\left\{
\begin{array}{ll}
= 4, & 1\leq t \leq 3\\
\leq 5, & t \geq 4
\end{array}\right.$ \\ \hline

$3$ & $\left\{
\begin{array}{llll}
= 3, & t=1 & &\\
= 4, & t\geq 2 & &
\end{array}\right.$& 
$\left\{
\begin{array}{llll}
= 4, & t=1 & &\\
= 5, & t \geq 2 &  &
\end{array}\right.$ & $\left\{
\begin{array}{ll}
= 4, & 1 \leq t \leq 2 \\
\leq 5, & t \geq 3 
\end{array}\right.$ \\ \hline

\end{tabu}}
\vspace{3mm}
\end{center}
\end{table}
%\hrule

\begin{table}[htb]
\begin{center}
\caption{The exact values and bounds for $\chi_p(P_{\ell t} \Diamond_{\ell} G)$, where $G$ is a cycle, continued.}\label{tabela3}

\vspace{3mm}
{\tabulinesep=1mm
\begin{tabu}{|c|*{4}{c|}}\hline
\backslashbox[20mm]{$\ell$}{$G$}  & $C_6$ & $C_{4s+2}$ & $C_3$ & $C_{4s+3}$\\ \hline

$2$ & $\left\{
\begin{array}{ll}
 = 4, & 1 \leq t \leq 2\\
\leq 5, & t \geq 3
\end{array}\right.$ & $\left\{
\begin{array}{ll}
 = 4, & 1 \leq t \leq 2\\
\leq 5, & t \geq 3
\end{array}\right.$ & $\left\{
\begin{array}{ll}
= 3, & t=1\\
= 4, & 2 \leq t \leq 3\\
\leq 5, & t\geq 4
\end{array}\right.$ & $\left\{
\begin{array}{ll}
= 4, & 1 \leq t \leq 2\\
\leq 5, & t\geq 3
\end{array}\right.$ \\ \hline

$3$ & $\left\{
\begin{array}{ll}
= 4, & 1 \leq t \leq 2\\
\leq 5, & t \geq 3
\end{array}\right.$ & $\left\{
\begin{array}{ll}
= 4, & 1 \leq t \leq 3\\
\leq 5, & t \geq 4
\end{array}\right.$ & $\left\{
\begin{array}{ll}
= 3, & t =1\\
= 4, & 2 \leq t\leq 3\\
\leq 5, & t \geq 4
\end{array}\right.$ & $\left\{
\begin{array}{ll}
= 4, & 1 \leq t \leq 2\\
\leq 5, & t\geq 3
\end{array}\right.$ \\ \hline

\end{tabu}}
\vspace{3mm}
\end{center}
\end{table}

By Proposition \ref{prop:Cn}, the results for $P_{\ell t} \Diamond_{\ell} C_n$ are equivalent to the results for $P_{(n-\ell+2)t} \Diamond_{n-\ell+2} C_n$. Hence, we have the following result.

\begin{corollary}\label{ref:SummaryCorollary} Let $\ell\in\{2,3, n-1,n\}$,  $t\geq 1$, and $n\geq 3$ be integers. Then,
    $$\chi_p(P_{\ell t} \Diamond_{\ell} C_n) \leq 6.$$
\end{corollary} 

Corollary \ref{ref:SummaryCorollary} shows that such a path-aligned cycle product, like paths, constitutes a graph class whose packing chromatic number remains bounded by a constant, even as the diameter increases. In addition, we may conjecture that, for every value of $\ell$, five colors always suffice for the packing coloring of path-aligned cycles. We state this formally in Conjecture \ref{conj:pathaligned}.

\begin{conj} \label{conj:pathaligned} Let $n\geq 3$, $t\geq 1$, and $4 \leq \ell \leq n$ be integers. Then,
    $$\chi_p(P_{\ell t} \Diamond_{\ell} C_{n}) \leq 5,$$
possibly except for $C_5$.
\end{conj}

In Theorem \ref{thm:partialconj}, we prove that Conjecture \ref{conj:pathaligned} holds for $\ell = 4$ and cycles $C_{4s+1}$ for $s\geq 1$.

\begin{theorem} \label{thm:partialconj}
    Let $s \geq 1$ and $t \geq 1$ be integers. Then,
    $$\chi_p(P_{4t} \Diamond_4 C_{4s+1})\leq 5.$$
    \label{thm:l4}
\end{theorem}
\begin{proof}
The case of $s=1$ is a direct consequence of Proposition \ref{prop:Cn} and Theorem \ref{thm:l3C5}. Now, let $s\geq 2$.

    Let $(a), (b),$ and $(c)$ denote the patterns $({\bf 1}\ 4\ 3\ 1\ 2 \ 1\ [3\ 1\ 2 \ 1]^*\ \overleftarrow{{\bf 3\ 1\ 5}})$, $({\bf 1}\ 4\ 1\ 2 \ 3 \ 1 \ [1\ 2 \ 3\ 1]^*\ \overleftarrow{{\bf 5\ 1\ 2}})$, and $({\bf 1}\ 4\ 3\ 1\ 2 \ 1\ [3\ 1\ 2 \ 1]^*\ \overleftarrow{{\bf 3\ 1\ 2}})$, respectively. For $t=1$, use the pattern $(c)$. For $t=2$ and $t=3$, use the patterns $(a)(b)$ and $(a)(b)(c)$, respectively. For $t \geq 4$, we repeat the $(a)(b)(c)$ pattern $\lfloor t/4 \rfloor$ times and add $(a)$ or $(a)(b)$ for $i=1,2$, respectively, for $t=3k+i$ for some integer $k$. In all cases, one can verify that the patterns ensure feasible 4-or 5- packing colorings.
\end{proof}

%Note that the Theorem \ref{thm:l4} confirms partialy the conjecture for $\ell = 4$. \textcolor{red}{Sibel: when we write the sentence before theorem, here we can start from the second sentence directly. Like In addition, we observe...} 
Furthermore, observe that the pattern for the packing $5$-coloring of the path $P_{3t}$ in $P_{3t} \Diamond_3 C_{4s+1}$, where $t \geq 3$ and $s \geq 2$, as presented in the proof of Theorem~\ref{thm:l3C4s+1}, coincides with the pattern for the packing $5$-coloring of the path $P_{4t}$ in $P_{4t} \Diamond_4 C_{4s+1}$, given in the proof of Theorem~\ref{thm:l4}. Indeed, in both cases, we have the pattern $1\ 5\ 1\ 3 \ 1 \ 2 \ 1 \ 5 \ 1 \ 2 \ 1 \ 3 \ldots$. This implies that, in both settings, the packing coloring of the path can be extended to $P_{\ell t} \Diamond_{\ell} C_{4s+1}$, for $t \geq 3$, $s \geq 2$, and $\ell \in \{3,4\}$, without introducing any additional colors.

This naturally raises the question of whether such a coloring of the path can be extended to the cycles $C_{4s+1}$ so as to yield a feasible packing coloring also for other values of~$\ell$ and for cycles of different lengths.

\begin{question}
Is there a coloring pattern for $P_{\ell t}$ in $P_{\ell t} \Diamond_{\ell} C_n$, for $n \geq 3$, that can be extended to a packing coloring of $P_{\ell t} \Diamond_{\ell'} C_n$ for every $n \geq \ell' \geq \ell$, without introducing any additional colors?
\end{question}

\subsection{Path-aligned complete products}
\label{subsec:complete}
For path-aligned complete graphs, we will define a coloring pattern for the path, and, at the same time, specify the pattern for the complete graphs. For example, the coloring of $P_6 \Diamond_2 K_4$ depicted in Figure \ref{fig:P6_K4} will be denoted as $({\bf 3}\ 2\ 4\ {\bf 1})({\bf 5}\ 2\ 6\ {\bf 1})({\bf 3}\ 2\ 4\ {\bf 1})$. Analogously to the case of packing coloring patterns for path-aligned cycles, the numbers in bold indicate the colors assigned to the vertices on the path.

%, while unbold numbers denote colors assigned to vertices only in the consecutive complete graphs. 

\begin{prop}\label{prop:Kn}
    %Let $t \ge 2$ and $n \ge 1$ be integers. Then, $P_{lt}\Diamond_l K_n$ is isomorphic to $P_{(n-l+2)t}\Diamond_{n-l+2} K_n$ for any integer $l\ge 2$.

    Let $t \ge 2$ and $n \ge 3$ be integers. Then, $P_{\ell t}\Diamond_{\ell} K_n$ is isomorphic to $P_{\ell^{'}t}\Diamond_{\ell^{'}} K_n$ for any integers $2 \leq \ell, \ell^{'} \le n$.
\end{prop}
\begin{proof}
 %The argument follows the same line of reasoning as in the proof of Proposition~\ref{prop:Cn}, wherein a Hamiltonian cycle $C_n$ is taken in each copy of $K_n$ and we consider .

 The argument follows the analogous line of reasoning as in the proof of Proposition~\ref{prop:Cn}. Now, we start from $P_{2t}\Diamond_{2} K_n$ and to show isomorphism to $P_{\ell t}\Diamond_\ell K_n$, i.e. the case where we have $\ell$ common vertices with the main path, we proceed as follows. For each copy of $K_n$ in $P_{2t}\Diamond_{2} K_n$, choose any $\ell -2$ vertices not lying on the path $P_{2t}$ and take them together with the two vertices already on the path. These $\ell$ vertices then serve as the common vertices with the path; repeating this choice for every copy and relabeling along the path yields $P_{\ell t}\Diamond_\ell K_n$, as desired. Since the graph isomorphism relation is transitive, it follows that there exists an isomorphism between $P_{\ell t}\Diamond_{\ell} K_n$ and $P_{\ell ' t}\Diamond_{\ell'} K_n$ for any pair of integers $\ell, \ell '$ satisfying $2 \leq \ell, \ell^{'} \le n$.
\end{proof}

Notice that in path-aligned complete product graphs, the diameter of the graph depends only on the path length; that is, $diam(P_n\Diamond_{\ell}K_t)=diam(P_n) = n - 1$ for all $\ell, t, n$. Furthermore, since $C_3$ is isomorphic to $K_3$, we have the same result as in Theorem \ref{thm:pathAlignedCycle}.

\begin{corollary}\label{cor:K_3}
Let $t\geq 1$ and $\ell \in \{2,3\}$ be integers. Then,     $$\chi_p(P_{2t}\Diamond_{\ell} K_3)\leq\left\{
\begin{array}{ll}
3, & t=1,\\
4, & 2 \leq t \leq 3,\\
5, & t\geq 4.
\end{array}\right.
$$
In addition, the equality holds for $t\in\{1,2,3\}$.     
\end{corollary} 

We now extend the result presented in Corollary \ref{cor:K_3}, first to $K_4$ in Theorem \ref{thm_K4}, and subsequently to $K_5$ in Theorem \ref{thm_K5}.

\begin{theorem} \label{thm_K4}
Let $t\geq 1$ and $2 \leq \ell \leq 4$ be integers. Then, 
    $$\chi_p(P_{2t}\Diamond_\ell K_4)\leq\left\{
\begin{array}{ll}
4, & t=1,\\
6, & t=2, 3,\\
8, & t\geq 4.
\end{array}\right.
$$
In addition, the equality holds for $t\in\{1,2,3\}$.
\end{theorem}

\begin{proof}
   By Proposition \ref{prop:Kn}, it is sufficient to prove the result for $\ell=2$. Let $(a)$, $(b)$, and $(c)$ denote the patterns $({\bf 3}\ 2\ 4\ \bf 1)$, $({\bf 5}\ 2\ 6\ \bf 1)$, and $({\bf 7}\ 2\ 8\ \bf 1)$, respectively. For $t=1,2,3,4$, the pattern is $(a)$, $(a)(b)$, $(a)(b)(a)$, and $(a)(b)(a)(c)$, respectively. For $t \ge 5$, if $t \equiv 0 \pmod{4}$, we repeat the $(a)(b)(a)(c)$ pattern $t/4$ times. Otherwise, we repeat the $(a)(b)(a)(c)$ pattern $\lfloor t/4 \rfloor$ times and add $(a)$, $(a)(b)$ or $(a)(b)(a)$ for $t \equiv 1,2,3 \ (\bmod 4)$, respectively. One can verify that these patterns constitute feasible packing colorings. Besides, it is easy to verify that $\chi_p(P_{2t}\Diamond_2 K_4) = 4$ for $t=1$. Let us now consider $k=2$. As we have two copies of $K_4$, any color can be used at most twice. In addition, by Observation \ref{obs:diam}, any color greater than 2 can be used at most once since $diam(P_4 \Diamond_2 K_4)=3$. Hence, we have that $\chi_p(P_{4}\Diamond_2 K_4) \ge 6$. Since $P_4 \Diamond_2 K_4$ is a subgraph of $P_6 \Diamond_2 K_4$, $\chi_p(P_{6}\Diamond_2 K_4) \ge 6$. %Since we have presented packing 6-colorings for these path-aligned complete graphs, they are optimal. 
    Hence, $\chi_p(P_{2t}\Diamond_2 K_4) = 6$ for $t=2,3$ (see Figure \ref{fig:P6_K4}).
\end{proof}

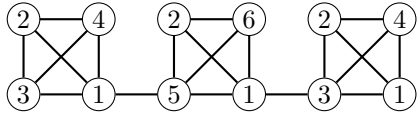
\begin{figure}[H]
\centering
\begin{tikzpicture}[scale=1.0, every node/.style={circle, draw=black, fill=white, inner sep=1.5pt}]

% Wierzchołki głównej ścieżki (10)
\foreach \i/\x/\col in {
  0/0/3, 1/1/1, 2/2/5, 3/3/1, 4/4/3,
  5/5/1
} {
  \node (v\i) at (\x,0) {\col};
}

% Krawędzie ścieżki głównej
\foreach \i in {0,...,4} {
  \pgfmathtruncatemacro{\j}{\i + 1}
  \draw[thick] (v\i) -- (v\j);
}

% Cykl C4 dla co drugiej krawędzi ścieżki (i parzyste)
\foreach \i/\x/\labB in {
  0/0/4,
  2/2/6,
  4/4/4}{

  \pgfmathtruncatemacro{\j}{\i + 1}
  \coordinate (a\i) at (\x,1);             % góra nad v_i
  \coordinate (b\i) at ({\x + 1},1);       % góra nad v_{i+1}
  
  % Rysuj cykl (prostokąt: v_i - a - b - v_{i+1} - v_i)
  \draw[thick] (v\i) -- (a\i) -- (b\i) -- (v\j);
  \draw[thick] (v\i) -- (b\i);
  \draw[thick] (a\i)  -- (v\j);
  % Wierzchołki cyklu z kolorami
  \node at (a\i) {2};
  \node at (b\i) {\labB};
}

\end{tikzpicture}
\caption{Packing 6-coloring of $P_{6} \mathbin{\Diamond_2} K_4$.}\label{fig:P6_K4}
\end{figure}

To obtain a feasible packing coloring pattern for $P_{2t} \Diamond_2 K_5$, we employed a computer-assisted approach based on integer linear programming (ILP). Specifically, we formulated the problem using the ILP framework proposed by Shao et al. \cite{shao2015ilp}, who introduced several alternative formulations for the packing coloring problem. Among these, we adopted the ILP-1 formulation and implemented it using the CPLEX optimization software. The formulation is described below.

 We consider a given graph $G=(V, E)$. For each vertex $v \in V(G)$ and each integer $i \in \{1, 2, \ldots, k\}$, let $x_{v,i}$ be a binary decision variable that equals $1$ if $v$ is labeled with $i$, and $0$ otherwise. Let $z$ be a continuous decision variable. The distance between any two vertices $u$ and $v$, defined as the number of edges on the shortest path between them, is denoted by $d(u,v)$. The problem of finding the packing chromatic number of a graph $G$ can be formulated as follows:

\begin{align}
\text{minimize } \quad & z \label{eq:obj} \\
\text{subject to:} \quad & \nonumber \\
 \sum_{i=1}^{k} x_{v,i} &= 1 && \forall v \in V(G) \label{eq:assign} \\
 x_{u,i} + x_{v,i} &\leq 1 && \forall u,v \in V(G),\ 1 \leq i \leq k,\ d(u,v) \leq i \label{eq:indep} \\
 i x_{v,i} &\leq z && \forall v \in V(G),\ 1 \leq i \leq k \label{eq:bound} \\
 x_{v,i} &\in \{0, 1\} && \forall v \in V(G),\ 1 \leq i \leq k \label{eq:binary}
\end{align}

The variable $z$ represents the highest color index employed in a feasible packing $k$-coloring. Accordingly, the objective function in \eqref{eq:obj} seeks to minimize $z$. Constraint \eqref{eq:assign} ensures that each vertex gets exactly one color, whereas Constraint \eqref{eq:indep} enforces that any two vertices within distance at most $i$ cannot be assigned the same color. Constraint \eqref{eq:bound} ensures that the variable $z$ refers to the index of the highest index color used. Finally, Constraint \eqref{eq:binary} ensures that $x_{v,i}$ are binary decision variables. 

The source code of the implementation together with some results can be found in \cite{melih2025ilp}. This implementation enabled us to obtain a feasible coloring pattern for $P_{2t} \Diamond_2 K_5$, which we use in Theorem \ref{thm_K5}.

From now on, $(x) \cdots (y)$ will denote the coloring pattern from $(x)$ to $(y)$ in alphabetical order in the Latin alphabet. For example, $(a)\cdots (d)$ will denote the pattern $(a)(b)(c)(d)$.

\begin{theorem}\label{thm_K5}
    Let $t \geq 1$ and $2 \leq \ell \leq 5$ be integers. Then,
    $$\chi_p(P_{\ell t}\Diamond_{\ell} K_5)\leq 14.$$
\end{theorem}
\begin{proof}
    By Proposition \ref{prop:Kn}, it is sufficient to prove the result for $\ell=2$. Let $(a), (b), (c), (d), (e), (f), (g), (h), (i), (j), (k)$, and $(l)$ denote the patterns $({\bf 5}\ 2\ 3 \ 10\ {\bf 1})$, $({\bf 9}\ 2\ 4 \ 11\ {\bf 1})$, $({\bf 3}\ 2\ 6 \ 8\ {\bf 1})$, $({\bf 7}\ 1\ 2 \ 4\ {\bf 5})$, 
$({\bf 1}\ 2\ 12 \ 13\ {\bf 3})$, $({\bf 1}\ 2\ 4 \ 6\ {\bf 10})$, $({\bf 1}\ 2\ 5 \ 8\ {\bf 3})$, $({\bf 11}\ 2\ 4 \ 7\ {\bf 1})$, $({\bf 9}\ 2\ 3 \ 6\  {\bf 1})$, $({\bf 5}\ 2\ 4 \ 14\ {\bf 1})$, $({\bf 3}\ 2\ 8 \ 12\ {\bf 1})$, and $({\bf 7}\ 2\ 4 \ 6\ {\bf 1})$, respectively.
\\ For $t \equiv 0 \pmod{12}$ we repeat the pattern 
%$(a)(b)(c)(d)(e)(f)(g)(h)(i)(j)(k)(l)$. 
$(a)\cdots (l)$. Otherwise, we repeat the 
$(a)\cdots (l)$ 
pattern $\lfloor t/6 \rfloor$ times, then add $(a)$, $(a)(b)$, $(a)(b)(c)$, $(a)\cdots (d)$, $(a)\cdots (e)$,  $(a)\cdots(f)$,  $(a)\cdots (g)$, $(a)\cdots (h)$,  
$(a)\cdots (j)$, 
$(a)\cdots (k)$,
and $(a)\cdots(l)$, for $t \equiv 1, 2, 3, 4, 5, 6, 7, 8, 9, 10, 11 \pmod{12}$, respectively.

%for $t \equiv p \pmod {12}$, where $p\in [11]$.

%For $t \equiv 0 \pmod{12}$ we repeat the pattern $(a)\cdots (l)$. Otherwise, we repeat the pattern $(a)\cdots(l)$ $\lfloor t/6 \rfloor$ times, then add the first $s$ patterns in $(a)\cdots(l)$  for $t \equiv s \pmod{12}$.

Now, we show that the given coloring is a feasible packing coloring as follows. First, note that the distance between vertices with color 1 is at least two. Next, observe that color 2 is assigned to vertices in two consecutive complete graphs that are not on the path, ensuring that the distance between any two such vertices is at least 3. Similarly, colors 3 and 4 are used in every second complete graph. The distance between any two vertices colored with 3 is at least 4, while the distance between any two vertices colored with 4 is at least 5.
For $i \in \{5,6\}$, if color $i$ appears in the $j^\text{th}$ complete graph along the path, its next occurrence is on the $(j+3)^\text{rd}$ complete graph. For $i \in \{7,8\}$, if color $i$ appears in the $j^\text{th}$ complete graph along the path, its next occurrence is on the $(j+4)^\text{th}$ complete graph. Consequently, for colors $5,6,7,8$ the distance between any two vertices assigned color $i$ is at least $i+1$.
Color 9 appears only on the vertices on the path $P_{2t}$, where the distance between any two such vertices is at least 10. It is easy to verify that any two vertices colored with color $i$, for $i\in\{10,11,12\}$, are at a distance of at least $i+1$. Furthermore, colors 13 and 14 appears in every $12^\text{th}$ complete graph. Therefore, the conditions for a valid packing coloring are satisfied in all cases.

In summary, for every color $i$, where $1 \leq i \leq 14$, the distance between any two vertices colored with $i$ is at least $i+1$. Hence, the coloring obtained according to the given pattern constitutes a proper packing 14-coloring.
\end{proof}

\subsubsection{Summarizing remarks}

In this subsection, we have analyzed the path-aligned products of complete graphs. First, we have given a general transformation result for $P_{\ell t}\Diamond_{\ell} K_n$ for any $n\geq 3$ in Proposition \ref{prop:Kn}. Then, we have provided upper bounds for $\chi_p(P_{\ell t}\Diamond_{\ell} K_n)$, for $n \in \{3, 4, 5\}$ and any feasible $\ell \geq 2$. For $\ell =1$, our computational experiments show that the number of colors for packing coloring of such graphs increases rapidly. We pose the following as an open question.
\begin{question}
    Is $\chi_p(P_{\ell t}\Diamond_{\ell} K_n)$ bounded by a constant for $\ell=1$?
\end{question}

%Moreover, we observe that the graph $P_t \Diamond_1 K_n$ is isomorphic to
%can be seen as
%is equivalent to the corona product $P_t \circ K_{n-1}$. The packing coloring of corona products is considered in the next section.

Besides, our computer experiments show that $\chi_p(P_{\ell t}\Diamond_{\ell} K_6) \le 22$ for $t \le 48$ and $2  \le \ell \le 6$. Then we pose the following conjecture:

\begin{conj}
    $\chi_p(P_{\ell t}\Diamond_{\ell} K_6) \le 22$ for $2  \le \ell \le 6$.
\end{conj}

Recall that the packing coloring problem is NP-complete even for trees \cite{fiala}, and consequently also for block graphs. Note that path-aligned complete product graphs constitute a subclass of block graphs. Therefore, the results presented in Section \ref{subsec:complete} identify an infinite family of graphs within the class of block graphs, where the packing coloring problem is generally NP-complete, for which the problem can be solved in polynomial time.

Furthermore, in the survey \cite{survey}, we have:

\begin{theorem}[Thm. 2.10]
If $G$ is a graph with $\chi_p(G)=\chi(G)$, then $\omega(G)=\chi(G)$.\label{thm:eq_chi}
\end{theorem}

Here we show that the converse is not true; counterexamples are path-aligned even cycles and path-aligned complete graphs. In particular, $\chi(P_{4}\Diamond_2 K_4) = 4 = \omega(P_{4}\Diamond_2 K_4)$ but $\chi_p(P_{4}\Diamond_2 K_4) =6 \ne 4$ by Theorem \ref{thm_K4}. It would be interesting to explore the graphs where the Theorem \ref{thm:eq_chi} is true.

\begin{question}
    What are the graphs $G$ such that if $\omega(G)=\chi(G)$, then $\chi_p(G)=\chi(G)$?
\end{question}

\section{Packing Coloring of Caterpillars} \label{sec:corona}

In this section, we study the packing chromatic number of the caterpillars, which can be seen as the generalized $p$-corona of paths. The first study addressing the packing coloring of corona graphs was conducted by Laïche et al. \cite{laiche2017packing}.
\begin{definition}
\emph{For a given integer $p \geq 1$, the}
 $p$-corona of a graph $G$ \emph{is the graph obtained from $G$ by adding $p$ degree-one
 neighbors to every vertex of $G$.}
\end{definition}  
The authors of \cite{laiche2017packing} determined the packing
 chromatic number of $p$-coronae of paths and cycles for every $p \geq 1$. In this section, we extend the results obtained for $p$-coronae of paths to caterpillars, which is a more general class of graphs since any number of $K_1$'s are attached to each vertex in a caterpillar.

\begin{prop}[\cite{laiche2017packing}]
The packing chromatic number of the corona  graph $P_n \circ K_1$ is given
 by
$$
\chi_p(P_n \circ K_1)=
\left\{
\begin{array}{ll}
2 & \text{ if } n=1,\\
3 & \text{ if } n\in\{2,3\},\\
4 & \text{ if } 4 \leq n \leq 9,\\
5 & \text{ if } n \geq 10.
\end{array}\right.
$$
\end{prop}

% \textcolor{blue}{So, we observe that $\chi_p(P_n \circ K_1) \le \chi_p(P_n )+2$. Note that the bound is the same even for long paths, which have large diameter. So, large diameter does not necessarily increase the upper bound.}
 
\begin{prop}[\cite{laiche2017packing}]
The packing chromatic number of the corona  graph $P_n \circ 2K_1$ is given
 by
$$
\chi_p(P_n \circ 2K_1)=
\left\{
\begin{array}{ll}
2 & \text{ if } n=1,\\
3 & \text{ if } n=2,\\
4 & \text{ if } n\in \{3,4\},\\
5 & \text{ if } 5 \leq n \leq 11,\\
6 & \text{ otherwise.}
\end{array}\right.
$$
\end{prop}

\begin{prop}[\cite{laiche2017packing}]
The packing chromatic number of the corona  graph $P_n \circ 3K_1$ is given by $$\chi_p(P_n \circ 3K_1)=
\left\{
\begin{array}{ll}
2 & \text{ if } n=1,\\
3 & \text{ if } n=2,\\
4 & \text{ if } n\in \{3,4\},\\
5 & \text{ if } 5 \leq n \leq 8,\\
6 & \text{ otherwise.}
\end{array}\right.
$$
\end{prop}

\begin{prop}[\cite{laiche2017packing}]
    The packing chromatic number of the corona  graph $P_n \circ pK_1$ is given
 by
$$\chi_p(P_n \circ pK_1)=
\left\{
\begin{array}{ll}
2 & \text{ if } n=1,\\
3 & \text{ if } n=2,\\
4 & \text{ if } n\in \{3,4\},\\
5 & \text{ if } 5 \leq n \leq 8,\\
6 & \text{ if } 9 \leq n \leq 34,\\
7 & \text{ otherwise.}
\end{array}\right.
$$\label{prop:cater}
\end{prop}

%In this paper we understand corona product of graphs more broadly, i.e. in the way given by Frucht and Harary \cite{frucht}.

%\begin{definition}[\cite{frucht}]
  %  \emph{The} corona \emph{of two graphs, $n_G$-vertex graph $G$ and $n_H$-vertex graph $H$, is a graph $G \circ H$ formed from one copy of $G$,  called the} center graph\emph{, and $n_G$ copies of $H$, named the} outer graph\emph{, where the $i$-th vertex of $G$ is adjacent to every vertex in the $i$-th copy of $H$.}
%\end{definition}

%Note that $P_k \Diamond_1 K_n$ can be seen as corona graph $P_k \circ K_{n-1}$.
\begin{definition}\label{def:caterpillar}
   A caterpillar graph $C(l;m_1,\ldots, m_l)$ \emph{of a length $l$, denoted also by $CT_l$, is a tree in which the removal of all pendant edges results in a chordless path $P_l$ called the \emph{backbone} of the caterpillar. In addition, we assume $V(P_l)=\{v_1,\ldots,v_l\}$ with $\{v_i,v_{i+1}\} \in E(P_l)$ for each $i\in \{1,\ldots,l-1\}$, and each $v_i$ is incident to $m_i$ leaves (cf. Fig.~\ref{fig:caterpillar_ex}).}
\end{definition}

%Note that due to the definition, the length of a caterpillar denotes the number of vertices in its backbone.

\begin{figure}[H]
\centering
\begin{tikzpicture}[scale=1.0, every node/.style={circle, draw=black, fill=white, inner sep=1.5pt}]

% Wierzchołki głównej ścieżki (10)
\foreach \i/\x/\col in {
  0/0/\textcolor{white}{\small{$v_0$}}, 1/1.2/\small{$v_1$}, 2/2.4/\small{$v_2$}, 3/3.6/\small{$v_3$}, 4/4.8/\small{$v_4$},
  5/6/\textcolor{white}{\small{$v_5$}}
} {
  \node (v\i) at (\x,0) {\col};
}

% Krawędzie ścieżki głównej
\foreach \i in {0,...,4} {
  \pgfmathtruncatemacro{\j}{\i + 1}
  \draw[thick] (v\i) -- (v\j);
}
\foreach \i/\x in {
  1/1.2,
  2/2.4}{

  \pgfmathtruncatemacro{\j}{\i + 1}
  \node (a\i) at (\x,1) {\textcolor{white}{\tiny{$a_i$}}};
  %\coordinate (a\i) at (\x,1);             % góra 
  %\node at (a\i) {\textcolor{white}{\tiny{$a_i$}}};
  \draw[thick] (v\i) -- (a\i);
}
\foreach \i/\x in {
  1/1}{  
  \node (b\i) at (\x - 0.4,1) {\textcolor{white}{\tiny{$b_i$}}};
   \node (c\i) at (\x + 0.8,1) {\textcolor{white}{\tiny{$c_i$}}};
  \draw[thick] (v\i) edge (b\i);
  \draw[thick] (v\i) -- (c\i);
}
\end{tikzpicture}
\caption{The caterpillar $C(4;4,1,0,1)$ with the backbone $v_1v_2v_3v_4$ of length 4.}\label{fig:caterpillar_ex}
\end{figure}
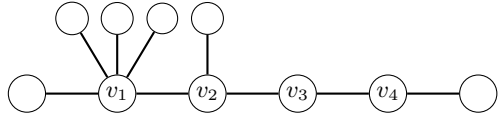

Sloper \cite{sloper} proved that eccentric chromatic number of any caterpillar $C$ are bounded from the above by 7. Since every eccentric coloring is also a packing coloring with the additional requirement that each color must be larger than the eccentricity of the corresponding vertex, the result in \cite{sloper} also implies that
$$\chi_p(C)\leq 7.$$

This result can be improved by using Proposition \ref{prop:cater} as follows.
\begin{corollary}
    Let $C=C(l;m_1,\ldots,m_l)$ be a caterpillar of length $l$. Then $\chi_p(C)\leq 6$ for $l\leq 34$, and $\chi_p(C)\leq 7$ otherwise. In both cases, the bounds are tight.
\end{corollary}
    
Goddard et al. \cite{goddard2008broadcast} provided a polynomial-time recognition algorithm for graphs $G$ with $\chi_p(G)\leq 3$. In this section, we provide a structural characterization of caterpillars with $\chi_p(G)\leq 3$, implying a polynomial-time recognition algorithm for caterpillars in particular. 

\subsection{Characterization of caterpillars with $\chi_p(CT_l)=3$}\label{sec:char3}

First, we define families of caterpillars to be used in this section as follows.
\begin{definition}
~~
    \begin{itemize}
        \item $\mathcal{G}_1=C(4k; m_1, m_2, 0,m_4,0, m_6, \ldots,0,m_{4k})$ for any $k\geq 1$,
        \item $\mathcal{G}_2=C(4k+1; 1, m_2,0,m_4,\ldots,0,m_{4k},m_{4k+1})$ for any $k\geq 1$,
        \item $\mathcal{G}_3=C(4k+1; m_1, 0,m_3,\ldots, 0,m_{4k+1})$ for any $k\geq 1$,
    \item $\mathcal{G}_4=C(4k+2; m_1, m_2,0,m_4,\ldots, 0,m_{4k+2})$ for any $k\geq 0$,
    \item $\mathcal{G}_5=C(4k+3; m_1, m_2, 0, m_4, \ldots, 0, m_{4k+2},m_{4k+3})$ for any $k\geq 1$,
    \item $\mathcal{G}_6=C(3; m_1,m_2,1)$,
    \item $\mathcal{G}_7=C(4k+3;m_1,0,m_3,\ldots, 0,m_{4k+3})$ for any $k\geq 0$,
    \end{itemize}
    where ``$\ldots$" means alternating subsequence of 0 and $m_i$'s.
\end{definition}

\begin{theorem} \label{thm:caterpillar}
Let $G$ be a caterpillar $C(l;m_1,m_2,\ldots,m_l)$ with $l\geq 2$, $m_1,m_l \geq 1$, and $m_2,\ldots,m_{l-1}\geq 0$. Then, $\chi_p(G)=3$, if and only if $G \in \cup_{i=1}^{7}\mathcal{G}_i$. 

\end{theorem}
\begin{proof}
    The proof is based on the observation that the colors assigned to $v_1$ and $v_2$ largely constrain the coloring of the remaining vertices in the caterpillar $G$.

 ($\Rightarrow$) Let $G$ be a caterpillar of the form $C(l;m_1,m_2,\ldots,m_l)$ with $l\geq 2$, $m_1,m_l \geq 1$, and $m_2,\ldots,m_{l-1}\geq 0$ such that $\chi_p(G)=3$. Then, there are six possibilities to assign colors to $v_1$ and $v_2$. We will now analyze these six cases separately. In each case, we start by coloring the vertices $v_1$ and $v_2$. Then, we consider how this coloring can be extended to the remaining vertices of the backbone and to the leaves of $G$. In general, we assume that the vertices of the backbone that receive a color different from 1 may have neighbors of degree 1 corresponding to $m_i \geq 0$, for $2\leq i\leq l-1$, and $m_i \geq 1$, for $i=1$ or $i=l$. Vertices $v_i$ colored 1 may have adjacent leaves only in the case of $i=1$ or $i=l$, and in these cases exactly one leaf is allowed.
~~\\

\emph{Case 1:} Let $c(v_1)=1$ and  $c(v_2)=2$. 

First, suppose that $v_1$ has at least two leaves. Then, the leaves together with $v_2$ would imply that $\chi_p(G) \ge 4$, a contradiction. Moreover, by the definition of a caterpillar, $v_1$ is adjacent to at least one leaf; hence, $v_1$ is adjacent to exactly one leaf. Since $c(v_2)=2$, this leaf has to be colored with $3$. Then, $c(v_3)=1$ and $c(v_4)=3$ are the only possible assignments. It is easy to observe that the coloring pattern for the backbone becomes $[1\ 2\ 1\ 3]^{*}$. 
%\textcolor{blue}{It is easy to observe that the coloring pattern for the backbone becomes $1\ 2\ [1\ 3\ 1\ 2]^{*}$.}
Indeed, the only undetermined color in this sequence is a color of the neighbor of the vertex colored with 3; that is, the color of vertices $v_{4s+1}$, for $s\geq 1$. Here, the vertex can be colored with 2 or with 1. However, if color 2 is assigned to $v_{4s+1}$, then color 1 must be used for $v_{4s+2}$, which in turn necessitates the use of color 4 for the remaining vertices in the backbone. Therefore, this case is possible only at the end of the backbone where $l=4k+1$.
Hence, let us assume that we have colored $l - (l \bmod 4)$ vertices of the backbone of $G$ with the pattern $[1\ 2\ 1\ 3]^{*}$. Now, we need to consider how to color the rest of $l \bmod 4$ vertices of the backbone. 
%\textcolor{blue}{Hence, let us assume that we have colored $l - ((l-2) \bmod 4)$ vertices of the backbone of $G$ with the pattern $1\ 2\ [1\ 3\ 1\ 2]^{*}$. Now, we need to consider how to color the rest of $(l-2) \bmod 4$ vertices of the backbone.}

\underline{Let $l=4k$:} Here, the pattern $[1\ 2\ 1\ 3]^{*}$ colors all the vertices of the backbone. Hence, $G \in \mathcal{G}_{1}$.

%\textcolor{blue}{\underline{Let $l=4k$:} Here we have two uncolored vertices in the backbone and the only possible extension of the pattern to those vertices is to assign color 1 to $v_{l-1}$ and color 3 to $v_{l}$. In this case, we have $G\in \mathcal{G}_4$. }

\underline{Let $l=4k+1$:} The only vertex remaining uncolored in the backbone is $v_{4k+1}$, which may be assigned color 1 or 2. If color 1 is assigned to $v_{4k+1}$, then it may have only one adjacent leaf, which must be assigned color 2, resulting in a graph $G\in \mathcal{G}_2$; otherwise, the resulting graph is again in  $G\in \mathcal{G}_2$, where multiple adjacent leaves are permissible for $v_{4k+1}$.

%\textcolor{blue}{\underline{Let $l=4k+1$:} Since $c(v_{4k})=3$, it follows that $c(v_{4k+1})=1$, $c(v_{4k+2})=2$, and $c(v_{4k+3})=1$. Note that, the leaf adjacent to $v_{4k+3}$ must be assigned color 3. Indeed, the distance of this leaf to $v_{4k}$, which is also colored 3, is four; hence, the resulting coloring constitutes a feasible packing coloring. So, $G \in C(4k+3; 1, m_2, 0, m_4, 0, m_6, \ldots, 0, m_{4k+2},1)$, for any $k\geq 1$, and therefore belongs to the family $\mathcal{G}_5$.}

\underline{Let $l=4k+2$:} Now, we have two uncolored vertices in the backbone and the only possible extension of the pattern to those vertices is to assign color 1 to $v_{4k+1}$ and color 2 to $v_{4k+2}$.  
In this case, we have $G\in \mathcal{G}_4$.

\underline{Let $l=4k+3$:} If $k=0$, then we color the vertices of the backbone with colors $1,2,1$, consecutively. Note that, this is the only possible coloring since we have exactly one leaf adjacent to $v_1$ that is colored with 3. Moreover, $v_3$ has to be adjacent to exactly one leaf that is also assigned color 3. Therefore, $G \in C(3;1,m_2,1)$ which is in the family $\mathcal{G}_6$. 
Now, let $k\geq 1$. Since $c(v_{4k})=3$, it follows that $c(v_{4k+1})=1$, $c(v_{4k+2})=2$, and $c(v_{4k+3})=1$. Note that, the leaf adjacent to $v_{4k+3}$ must be assigned color 3. Indeed, the distance of this leaf to $v_{4k}$, which is also colored 3, is four; hence, the resulting coloring constitutes a feasible packing coloring. So, $G \in C(4k+3; 1, m_2, 0, m_4, 0, m_6, \ldots, 0, m_{4k+2},1)$, for any $k\geq 1$, and therefore belongs to the family $\mathcal{G}_5$.

~~\\
\emph{Case 2:} Let $c(v_1)=1$ and  $c(v_2)=3$.

Here, we follow a similar approach. The vertex $v_1$ may have only one leaf that must be assigned color 2. To extend the coloring to the backbone, we are required to use the coloring pattern $[1\ 3\ 1\  2]^*$. Indeed, the only undetermined color in this sequence is the color of the neighbor of the vertex colored with 3, namely, the color of vertices $v_{4s+3}$, for $s\geq 1$. Here, the vertex may be colored with 1 or 2. However, if color 2 is assigned to $v_{4s+3}$, then $v_{4s+4}$ must be colored with 1, which subsequently necessitates the use color 4 for the remaining vertices along the backbone. Therefore, such a situation can occur only at the end of the backbone, and moreover, only for graphs $G$ with $l=4k+3$.
\\Hence, let us assume that the initial $l - (l \bmod 4)$  vertices of the backbone of $G$ are colored using the pattern $[1\ 3\ 1\ 2]^*$. In the following, we consider how to extend the coloring for the remaining $l \bmod 4$ vertices of the backbone.

\underline{Let $l=4k$:} Here, the backbone is colored by the pattern $[1\ 3\ 1\ 2]^*$, yielding $C(4k;1,m_2,0,m_4,\ldots, 0, m_{4k}) \in \mathcal{G}_{1}$.

\underline{Let $l=4k+1$:} The only vertex  in the backbone remaining uncolored is $v_{4k+1}$, which must be assigned color 1, and it may have at most one adjacent leaf, which must be assigned color 3 (the distance to the nearest vertex colored with 3, that is to $v_{4k-2}$, is 4, yielding a feasible packing coloring). The family of caterpillars admitting this coloring can be described as follows: $C(4k+1; 1, m_2, 0, m_4, \ldots, m_{4k}, 1)$, for any $k\geq 1$, which is a member of $\mathcal{G}_2$.

\underline{Let $l=4k+2$:} Here we have two uncolored vertices in the backbone and the only possible extension of the pattern to those vertices is to assign color 1 to $v_{4k+1}$ and color 3 to $v_{4k+2}$. In this case, the family of caterpillars is described as $C(4k+2;1,m_2,0,m_4,0,m_6, \ldots,0,m_{4k+2}) \in  \mathcal{G}_4$, for any $k\geq 0$.

\underline{Let $l=4k+3$:} The three uncolored vertices may be assigned the colors $1, 3, 1$ or $1,3, 2$, respectively. It is easy to see that the latter coloring permits a larger number of leaves adjacent to $v_l$. In either extension, we have $G\in \mathcal{G}_5$, provided that $k\geq 1$.  If $k=0$, that is, if $l=3$, we likewise apply the coloring  $1, 3, 1$ or $1,3, 2$ to the vertices of the backbone; however, accounting for symmetry, this results in caterpillars belonging to the family $\mathcal{G}_6$.

~~\\
\emph{Case 3:} Let $c(v_1)=2$ and  $c(v_2)=1$. 

First, note that unlike in the previous two cases, vertex $v_1$ may have any number of leaves attached. To extend the coloring of $v_1$ and $v_2$ to the entire backbone we obtain the pattern $[2\ 1\ 3\ 1]^*$. Again, the only undetermined color in this sequence is the color of the neighbor of the vertex colored with 3, that is, the color of the vertices $v_{4s}$, for $s\geq 1$. Here, the vertex may be colored with 1 or 2. Nevertheless, if color 2 is assigned to $v_{4s}$, then $v_{4s+1}$ must be colored with 1, which subsequently necessitates the use of color 4 for the remaining vertices along the backbone. Thus, such a situation can occur only at the end of the backbone and, moreover, only for graphs $G$ with $l=4k$, where $k\geq 1$. Therefore, $v_{4s}$ must be assigned color 1, for any $ 1 \leq s\leq k -1$.
Hence, let us assume that we have colored $l - (l \bmod 4)$ initial vertices of the backbone of $G$ using the pattern $[2\ 1\ 3\ 1]^*$. In the following, we extend the coloring for the remaining $l \bmod 4$ vertices of the backbone.

\underline{Let $l=4k$:} Here, all vertices of the backbone are colored using either the pattern $[2\ 1\ 3\ 1]^*$ or $[2\ 1\ 3\ 1]^*[2\ 1\ 3\ 2]$. Both situations, taking symmetry into account, yield a caterpillar belonging to the family $\mathcal{G}_{1}$.

\underline{Let $l=4k+1$:} The vertex $v_{4k+1}$ must be assigned color 2 and may have an arbitrary number of adjacent leaves, which must be assigned color 1. Hence, $G\in \mathcal{G}_3$.

\underline{Let $l=4k+2$:} In this case, two vertices of the backbone remain uncolored, and the only possible extension of the pattern assigns color 2 to $v_{4k+1}$ and color 1 to $v_{4k+2}$. Note that $v_l$ may have only one adjacent leaf, which must be assigned color 3. The distance to the other vertices colored with 3 is, once again, consistent with a feasible packing coloring. Thus, $G \in C(4k+2; m_1, 0, m_3, 0, \ldots, m_{4k+1}, 1)$, for any $k\geq 0$, which belongs to the family symmetric to $\mathcal{G}_4$. 

\underline{Let $l=4k+3$:} The three uncolored vertices can only be assigned colors $2, 1, 3$, consecutively. In this case, $G\in \mathcal{G}_7$.

~~\\
\emph{Case 4:} Let $c(v_1)=2$ and $c(v_2)=3$. 

In this case, the extension of the backbone coloring differs substantially from that in the preceding three cases. Note that starting with colors $2$ and $3$, the sequence must continue with colors $1, 2, 1, 3, \ldots$ and the first vertex with undetermined color is $v_7$; in general, $v_{4s+7}$, for $s\geq 0$, that is, the vertex adjacent to the one colored with $3$. At this point, vertex $v_{4s+7}$ may be assigned color 1 or color 2. However, if it is colored with 2, the backbone must terminate at this vertex. In other words, such a coloring is feasible only at the end of the backbone, which occurs when $l=4k+3$, $k\geq 1$. Hence, let us assume that $l - ((l-2) \bmod 4)$ initial vertices of the backbone of $G$ are colored using the pattern $2\ 3\ [1\ 2\ 1\ 3]^*$, and $(l-2)\bmod 4$ vertices of the backbone remain uncolored. We now extend the coloring to these vertices.

\underline{Let $l=4k$:} 
%In addition, let $k\geq 1$. 
In this case, two vertices remain uncolored, and the extension is uniquely determined: we assign color 1 to $v_{l-1}$ and color 2 to $v_l$. It follows that $G \in C(4k; m_1, m_2,0, m_4,\ldots, 0, m_l)$, which is in $\mathcal{G}_1$.

\underline{Let $l=4k+1$:} 
%In addition, let $k\geq 1$. 
Here, three vertices remain uncolored, and the only possible extension assigns the colors $1,2,1$, consecutively. Note that $v_l$ can have only one leaf, which must be assigned color 3. Again, the distance to the nearest vertex colored 3, namely $v_{l-3}$, is 4; hence, the coloring is a feasible packing coloring.
In this case, $G\in C(4k+1; m_1,m_2,0,m_4,\ldots,0,m_{4k},1)$, which is in $\mathcal{G}_2$.

\underline{Let $l=4k+2$:} Here, all vertices of the backbone of $G$ are colored using the pattern $2\ 3\ [1\ 2\ 1\ 3]^*$, implying that $G\in \mathcal{G}_4$. 

\underline{Let $l=4k+3$:} In this case, one vertex remains uncolored, namely the vertex $v_l$, and the vertex $v_{l-1}$ is assigned color 3. If $k=0$, that is, the length of the backbone is 3, the vertex $v_l$ must be colored 1, yielding the coloring $2, 3, 1$. Note that $v_l$ may have only one adjacent leaf, which must be assigned color 2. Finally, we obtain a caterpillar from the family $\mathcal{G}_6$. If $k\geq 1$, two possible colors  arise for $v_l$: (i) when $c(v_l)=1$, the vertex $v_l$ may have only a single adjacent leaf, which must be colored 2; or (ii) when $c(v_l)=2$, the vertex $v_l$ may have one or more adjacent leaves. In both cases, $G \in \mathcal{G}_5$. 

~~\\
\emph{Case 5:} Let $c(v_1)=3$ and $c(v_2)=1$. 

Here, we will extend the coloring from $v_1$ and $v_2$ to the entire backbone; in particular, we will show that the coloring pattern $3\ [1\ 2\ 1\ 3]^*$ must be used. Indeed, the only vertex with an undetermined color in this pattern is the neighbor of the vertex colored with 3. Either color 1 or 2 may be used to color this vertex. Nevertheless, color 2 is permissible only at the end of the backbone, that is, when $l=4k+2$. Thus, we assume that the first $l-((l-1) \bmod 4)$ vertices of the backbone of $G$ are colored according to the pattern $3\ [1\ 2\ 1\ 3]^*$. At this stage, $(l-1)\bmod 4$ vertices at the end of the backbone remain uncolored.

\underline{Let $l=4k$:} 
%Let $k\geq 1$. 
Three vertices remain uncolored, and colors $1,2,1$ must be used consecutively. Observe that $v_l$ may have at most one adjacent leaf, and this leaf must be assigned color 3. The distance to the nearest vertex colored  3, namely $v_{l-3}$, is four; hence, the coloring is feasible. Thus, $G \in C(4k;m_1,0,m_3,0, \ldots, 0, m_{4k-1},1)$, which is in $\mathcal{G}_1$, taking symmetry into account.

\underline{Let $l=4k+1$:} 
%And let $k\geq 1$. 
All vertices of the backbone of $G$ are colored with the pattern $3\ [1\ 2\ 1\ 3]^*$, and $G \in \mathcal{G}_3$.

\underline{Let $l=4k+2$:} 
%And let $k\geq 0$. 
Here, one vertex remains uncolored, namely the vertex $v_l$, while vertex $v_{l-1}$ is assigned color 3. There are two possibilities: $(i)$ $v_l$ can be colored 1, in which case it is adjacent to exactly one leaf which must be assigned color 2; or $(ii)$ $v_l$ can be colored 2, in which case it may have any number of adjacent leaves, with $m_l \geq 1$. In both cases, $G_4\in C(4k+2; m_1,0,m_3,0,\ldots,m_{4k+1},m_{4k+2})$, which is in $\mathcal{G}_4$, taking symmetry into account.

\underline{Let $l=4k+3$:} When $k=0$, the colors 3, 1, 2 are assigned to $v_1$, $v_2$, and $v_3$, respectively. Otherwise, the two uncolored vertices $v_{l-1}$ and $v_l$ must be colored with 1 and 2, respectively. Thus, in both cases, $G \in \mathcal{G}_7$.

~~\\
\emph{Case 6:} Let $c(v_1)=3$ and $c(v_2)=2$. 

Observe that such a coloring cannot be extended. In fact, if $v_3$ existed, it would need to be assigned color 1, and its neighbor, whether a leaf or the next vertex on the backbone, would necessarily require color 4, contradiction. 
Hence, $G\in C(2; m_1,m_2)$, which is in $\mathcal{G}_4$.

\vspace{0.3cm}
 ($\Leftarrow$) It remains to show that every graph in $\bigcup_{i=1}^7 \mathcal{G}_i$ has packing chromatic number 3. For each family of graphs in $\bigcup_{i=1}^7 \mathcal{G}_i$, we provide a corresponding coloring pattern for the backbone. The leaves adjacent to the vertices colored 2 or 3 are assigned color 1. When a leaf is adjacent to a vertex colored 1 (that is, to $v_1$ or $v_l$), we explicitly indicate its color.

\begin{quote}
    
\begin{itemize}
        \item $\mathcal{G}_1=C(4k; m_1, m_2, 0,m_4,0, m_6, \ldots,0,m_{4k})$, for any $k\geq 1$.

        Here, the coloring pattern for the backbone is $2\ 3\ [1\ 2\ 1\ 3]^{k-1}\ 1\ 2$.
        \item $\mathcal{G}_2=C(4k+1; 1, m_2,0,m_4,\ldots,0,m_{4k},m_{4k+1})$, for any $k\geq 1$.

        In this case, the backbone follows the coloring pattern $[1\ 2\ 1\ 3]^k \ 2$. Note that $v_1$ has only one adjacent leaf, which must be colored 3.
        \item $\mathcal{G}_3=C(4k+1; m_1, 0,m_3,\ldots, 0,m_{4k+1})$, for any $k\geq 1$.
    
        In this case, the coloring pattern of the backbone is $[3\ 1\ 2\ 1]^k\ 3$.
    \item $\mathcal{G}_4=C(4k+2; m_1, m_2,0,m_4,\ldots, 0,m_{4k+2})$, for any $k\geq 0$.
    
    The coloring pattern for the backbone is $2\ 3\ [1\ 2\ 1\ 3]^k$.
    \item $\mathcal{G}_5=C(4k+3; m_1, m_2, 0, m_4, \ldots, 0, m_{4k+2},m_{4k+3})$, for any $k\geq 1$.
    
    The coloring pattern for the backbone is $2\ 3\ [1\ 2\ 1\ 3]^k \ 2$.
    \item $\mathcal{G}_6=C(3; m_1,m_2,1)$.

    In this case, the coloring pattern $2\ 3\ 1$ is used for the backbone. Note that $v_3$ is adjacent to only one leaf, which must be assigned color 2.
    
    \item $\mathcal{G}_7=C(4k+3;m_1,0,m_3,\ldots, 0,m_{4k+3})$, for any $k\geq 0$.

    The coloring pattern for the backbone is $2\ 1\ 3\ [1\ 2\ 1\ 3]^k$.
    \end{itemize}

\end{quote}
 The proof is complete.
\end{proof}

  Goddard et al. proved in \cite{goddard2008broadcast} that any graph $G$ with $\chi_p(G) \le 3$ can be recognized in polynomial-time, whereas it is NP-complete to determine whether a given graph $G$ has $\chi_p(G) \le 4$. Our structural characterization in Theorem \ref{thm:caterpillar}  also implies a polynomial-time recognition algorithm for caterpillars with $\chi_p(G) \le 3$. It is therefore an interesting research direction to extend our characterization to caterpillars with $\chi_p(G) \le 4$, as this would also address the question of whether a polynomial-time recognition algorithm exists for caterpillars with $\chi_p(G) \le 4$, given that the problem is NP-complete in the general case. Hence, we pose the following open question:

  \begin{question}
      Can caterpillars with $\chi_p(G) \le 4$ be recognized in polynomial-time?
  \end{question}

\section*{Acknowledgements}
This work has been supported by the European Commission's Horizon Europe Research and Innovation programme through the Marie Skłodowska-Curie Actions Staff Exchanges (MSCA-SE) under Grant Agreement no.101182819 (COVER: (C)ombinatorial (O)ptimization for (V)ersatile Applications to (E)merging u(R)ban Problems) and by the Scientific and Technological Research Council of Türkiye (TÜBİTAK) under grant no.124F114.

\end{document}